\documentclass[a4paper,11pt,reqno]{amsart}

\date{\today}
\usepackage[latin1]{inputenc}

\usepackage{
amssymb,
tikz,
xcolor,
}
\usepackage{physics}
\usepackage[hidelinks]{hyperref}
\usepackage{bbm}
\usepackage{tikz-cd}    
\usepackage{geometry}
\newtheorem{thm}{Theorem}[section]
\newtheorem{lemma}[thm]{Lemma}
\newtheorem{proposition}[thm]{Proposition}

\theoremstyle{remark}
\newtheorem{remark}[thm]{Remark}

\newcommand{\R}{\mathbb{R}}
\newcommand{\C}{\mathbb{C}}
\newcommand{\Z}{\mathbb{Z}}

\newcommand{\cD}{\operatorname{\mathcal{D}}}
\newcommand{\cDmo}{\operatorname{\mathcal{D}_{m,\omega}}}
\newcommand{\cL}{\mathcal{L}}

\renewcommand\phi{\varphi}
\newcommand{\eps}{\varepsilon}

\renewcommand{\geq}{\geqslant}
\renewcommand{\leq}{\leqslant}

\newcommand{\Del}{\operatorname{\Delta}}

\DeclareMathOperator{\supp}{supp}

\renewcommand{\geq}{\geqslant}
\renewcommand{\leq}{\leqslant}

\newcommand\I{{\rm{i}}}

\newcommand\N{{\mathbb{N}}}
\newcommand{\be}{\begin{equation}}
\newcommand{\ee}{\end{equation}}

\title[Symmetric solutions for a 2D critical Dirac equation]{Symmetric solutions for a 2D critical Dirac equation}

\author[W. Borrelli]{William Borrelli}
\address[W. Borrelli]{Centro De Giorgi, Scuola Normale Superiore, Piazza dei Cavalieri 3, I-56100 , Pisa, Italy.}
\email{william.borrelli@sns.it}

\begin{document}

\begin{abstract}
In this paper we show the existence of infinitely many symmetric solutions for a cubic Dirac equation in two dimensions, which appears as effective model in systems related to honeycomb structures. Such equation is critical for the Sobolev embedding and solutions are found by variational methods. Moreover, we also prove smoothness and exponential decay at infinity.
\end{abstract}

\maketitle
{\footnotesize

\medskip

}


\section{Introduction}
\label{sec-intro}

\subsection{Motivation and main results}

This paper is devoted to the study of solutions of the following nonlinear massive Dirac equations
\be\label{criticaldirac}
(\cD+m\sigma_3-\omega)\psi=\vert\psi\vert^2\psi  \qquad\mbox{on}\quad\R^{2}\,,
\ee
where $\omega\in(-m,m)$ is a frequency in the spectral gap of the Dirac operator $\cD+m\sigma_3$, with $m>0$, and the nonlinearity is \emph{Sobolev-critical}. 
\medskip

Equation \eqref{criticaldirac} appears in the effective description of wave propagation in two-dimensional systems with the symmetries of a honeycomb lattice, under suitable assumptions. More precisely, If $V\in C^\infty(\R^2,\R)$ possesses the symmetries of a honeycomb lattice, the Schr\"odinger operator
\be\label{eq:hamiltonian}
H=-\Del+V(x)\,,\qquad x\in\R^2\,,
\ee
exhibit generically conical intersections (the so-called Dirac points) in its dispersion bands, as proved in \cite{FWhoneycomb}. The \emph{massless} (i.e., $m=0$) Dirac operator then appears as an effective operator describing, the dynamics of wave packets spectrally concentrated around those conical points \cite{FWwaves}. A mass term appears in the effective equation, when a perturbation breaking parity is added, as shown in \cite[Appendix]{FWhoneycomb}. Moreover, considering stationary solutions of the nonlinear Schr\"odinger equation 
\[
\I\partial_t u=H u+\vert u\vert^2u\,,
\]
with frequency corresponding to the conical crossing (at least formally) leads to an effective cubic nonlinearity of the form
\be\label{eq:generalcubic}
 \begin{pmatrix}
( \beta_1 |\psi_2|^2 + 2\beta_2 |\psi_1|^2)\psi_2 \\
( \beta_1 |\psi_1|^2 + 2\beta_2 |\psi_2|^2)\psi_1
\end{pmatrix}\,,\qquad \mbox{with $\psi=(\psi_1,\psi_2)^T$,}
\ee
as first computed in \cite{wavedirac}, where the parameters $\beta_1,\beta_2>0$ depend on the potential $V$ in \eqref{eq:hamiltonian}. In this paper we focus on the effective equation \eqref{criticaldirac} with a pure power nonlinearity, corresponding to the choice of parameters $\beta_1=1, \beta_2=1/2$, as it clearly leads to the same analytical difficulties as the general case. We mention the paper \cite{arbunichsparber}, where the validity of the effective cubic equation is addressed. Moreover, existence and qualitative properties of stationary solutions to the effective massless cubic Dirac equation have been studied in \cite{massless,borrellifrank}. Concerning the massive case \eqref{criticaldirac}, the existence of a solution of a particular symmetric form has been established in \cite{shooting} by dynamical systems arguments. In this paper we prove the existence of infinitely many (distinct) symmetric solutions, using variational methods, crucially exploiting the classification for the limit massless equation obtained in \cite{borrellifrank} (see Section \ref{sec:preliminaries}).
\smallskip

Critical Dirac equations have been studied  in connection with problem from conformal spin geometry, for which we refer the reader to \cite{ammannsmallest,spinorialanalogue,ammannmass,bartschspinorial,nadineconformalinvariant,nadineboundedgeometry,isobecritical,maalaoui,raulot} and references therein. We also mention that the case of coupled systems involving the Dirac operator and critical nonlinearities have also been recently studied in the literature, see \cite{Borrelli-Maalaoui-JGA2020,Maalaoui-Martino-JDE19}. From the point of view of analysis, those equation are conformally covariant or involve conformally covariant nonlinear terms, so that one has to deal with the associated loss of compactness, looking for a solution by variational methods. The required compactness analysis for our case is performed in Section \ref{sec:compactness}. The variational approach to nonlinear Dirac equations has been introduced in \cite{es} and has been subsequently widely employed, see for instance \cite{bartschspinorial,dingruf,dingwei,isobecritical}.
\medskip

In this paper we focus on the existence of symmetric solutions to \eqref{criticaldirac} of the following form
\begin{equation}\label{eq:ansatz}
\psi(x) = \begin{pmatrix}
\psi_1(x) \\
\psi_2(x) \end{pmatrix}
=
\begin{pmatrix}
v(r) e^{\I S\theta} \\
\I u(r) e^{\I (S+1)\theta}
\end{pmatrix} \,,\qquad x\in\R^2\,,S\in\Z\,,
\end{equation}
where $(r,\theta)\in(0,\infty)\times\mathbb{S}^1$ are polar coordinates in $\R^2$, and $u,v$ are real-valued functions.
\begin{remark}
Functions of the above form are the counterpart for the Dirac operator, of radial solutions for the Laplace equation. Indeed, while the laplacian commutes with rotations, this is not the case for the Dirac operator. More details about this property and its physical meaning can be found, for instace, in \cite{diracthaller} and references therein. 
\end{remark}
\begin{thm}\label{thm:main}
Let $S\in\Z$, $S\neq0$ and take $\omega\in (-m,m)$, with $m>0$. Then equation \eqref{criticaldirac} admits a non-trivial solution $\psi \in C^\infty(\R^2,\C^2)$ of the form \eqref{eq:ansatz}. Such solution vanishes at the origin, i.e. $\psi(0)=0$, and it
is exponentially localized, namely
\[
\vert\psi(r,\theta)\vert\leq C e^{-\frac{\sqrt{m-\omega}}{2}\,r}\,\qquad r>0,\theta\in\mathbb{S}^1\,,
\]
for some constant $C>0$.
\end{thm}
\begin{remark}
The solutions given by the above Theorem and that found in \cite{shooting} (for $S=0$) have the same exponential decay rate, but we do not know if such estimate is optimal. However, by \cite[Corollary 1.8]{CassanoDecay} one easily sees that solutions cannot decay at infinity faster than a gaussian. Indeed, in two dimensions such result holds in particular for Dirac equations of the form $(\cD+m\sigma_3)\psi+\mathbb{W}\psi=\omega\psi$, where $\mathbb{W}\in L^\infty_{loc}(\R^2,\C^2\times\C^2)$, and in our case it suffices to consider the scalar potential $\mathbb{W}:=\vert\psi\vert^2$.
\end{remark}

In \cite{shooting} solutions of the form \eqref{eq:ansatz} with $S=0$ were found by dynamical systems method, while in this paper we deal with $S\neq0$ using variational methods. We remark that our proof relies on the localization properties of the solutions of the limit equation \eqref{eq:limitequation} of the form \eqref{eq:ansatz}. Those functions corresponds to the blow-up profiles appearing in the variational procedure, and have been classified in \cite{borrellifrank}. For $S=0$, it turns out they are not square integrable, which prevents the application of the variational methods employed for $S\neq0$, while in this case solutions have a stronger decay at infinity.

\subsection{Outline of the paper}
The paper is organized as follows. In Section \ref{sec-intro} we state the main result of the paper and give an introduction to the problem under study. Section \ref{sec:preliminaries} contains some preliminary notions used in the sequel, while in Section \ref{sec:dual} we explain how to reformulate the problem in an equivalent way, using duality arguments. As explained through the paper, this approach allows to simplify the proof of the main results. The required compactness analysis is performed in Section \ref{sec:compactness}. Finally, we give the proof of the main Theorem in Section \ref{sec:proof}, which is divided into three different steps, for the convenience of the reader.
\medskip

\noindent{\bf Acknowledgements.}
The author is grateful to B. Cassano for bringing to his attention some results contained in \cite{CassanoDecay}.

The author is member of GNAMPA as part of INdAM and is supported by Centro  di  Ricerca  Matematica  \emph{Ennio  de  Giorgi}. 

\section{Preliminaries}\label{sec:preliminaries}
In this section we present some basic notions and recall results useful in the sequel.
\subsection{The operator} The \emph{Dirac operator} is a first order differential operator formally defined in two dimensions as
\begin{equation}
\mathcal{D}_{m}=\mathcal{D}+m\sigma_{3}:=-\I\sigma\cdot\nabla+m\sigma_{3}
\end{equation}
The constant $m>0$ is referred to as the `mass', as it usually represents such quantity in applications. In the above formula we use the notation $\sigma\cdot\nabla:=\sigma_{1}\partial_{1}+\sigma_{2}\partial_{2}$ and the $\sigma_{k}$ are the Pauli matrices
\begin{equation}\label{eq:pauli} \sigma_{1}:=\begin{pmatrix} 0 \quad& 1 \\ 1 \quad& 0 \end{pmatrix}\quad,\quad \sigma_{2}:=\begin{pmatrix} 0 \quad& -\I \\ \I \quad& 0 \end{pmatrix} \quad,\quad \sigma_{3}:=\begin{pmatrix} 1 \quad& 0 \\ 0 \quad& -1 \end{pmatrix}\,.\end{equation}

The operator $\mathcal{D}_{m}$ is a self-adjoint operator on $L^{2}(\mathbb{R}^{2},\mathbb{C}^{2})$, with domain $H^{1}(\mathbb{R}^{2},\mathbb{C}^{2})$ and form-domain $H^{1/2}(\mathbb{R}^{2},\mathbb{C}^{2})$. 

Moreover, since in Fourier domain $p=(p_{1},p_{2})$ the Dirac operator becomes the multiplication operator by the matrix 
$$ \widehat{\mathcal{D}}_{m}(p)=\begin{pmatrix}m \quad & p_{1}-\I p_{2}\\ p_{1}+\I p_{2}\quad &m \end{pmatrix}$$ 
then the spectrum is given by

\begin{equation}\label{eq:spectrum}
Spec(\mathcal{D}_{m})=(-\infty, -m]\cup[m, +\infty)
\end{equation}
where the gap is due to the mass term. The reader can find of the above mentioned results in \cite{diracthaller}.
\subsection{The functional}
Equation \eqref{criticaldirac} can be regarded as the Euler-Lagrange equation for the functional
\be\label{eq:action}
\cL_\omega(\psi)=\frac{1}{2}\int_{\R^2}\langle(\cD+m\sigma_3-\omega)\psi,\psi\rangle\,dx-\frac{1}{4}\int_{\R^2}\vert\psi\vert^4\,dx\,,
\ee
defined for $\psi\in H^{1/2}(\R^2,\C^2)$.

The above functional is strongly indefinite, that is, it is unbounded both from above and below, even modulo finite dimensional subspaces. This is due to the unboundedness of the spectrum \eqref{eq:spectrum}. This constitutes a difficulty for the application of variational methods, but several techniques have been introduced to deal with such situations (see for instance \cite{struwevariational} or \cite{els}). The principal difficulty in the application of variational methods to the search of critical points of \eqref{eq:action} is given by the fact that we have to deal with the Sobolev embedding 
\[
H^{\frac{1}{2}}(\R^2,\C^2)\hookrightarrow L^4(\R^2,\C^2)\,,
\]
whose compactness is prevented by the invariance by translation and scaling. The latter, in particular, implies that the embedding is not even locally compact and gives rise to the so called \emph{bubbling phenomenon}. This means that Palais-Smale sequences for the functional $\cL_\omega$ can concentrate peaking around some points in $\R^2$, preventing strong convergence in $L^4$ and thus in $H^{1/2}$. This phenomenon is common to variational problems involving Sobolev critical nonlinearities. We refer to \cite{CCcompactI,CCcompactII} for a general account on this kind of problems, in the framework of Concentration-Compactness theory and to \cite{isobecritical}, where the blow-up analysis has been carried out for a critical Dirac equation on a compact spin manifold. A similar result holds in our case, as explained in Section \ref{sec:compactness}.
\smallskip

In the sequel we will consider critical points of the functional restricted to the susbpace $E_S\subseteq H^{1/2}(\R^2,\C^2)$ of spinors of the form \eqref{eq:ansatz}, see Section \ref{sec:proof}. Such ansatz breaks the invariance by translation and partially simplifies the compactness analysis giving the origin as the only possibile blow-up point for Palais-Smale sequences.
\subsection{The limit equation}
The blow-up profiles (the so-called \emph{bubbles}) appearing in Palais-Smale sequences for $\cL_\omega$ are given by rescaled solutions $\Psi\in \mathring{H}^{1/2}(\R^2,\C^2)$ of the equation
\be\label{eq:limitequation}
\cD\Psi=\vert\Psi\vert^2\Psi\,,
\ee
which can be considered as the the \emph{limit equation} with respect to scaling. Indeed, at least formally, one can realize that considering the scaling
 \be\label{eq:scaling}
\varphi\mapsto\varphi_\delta:=\delta^{-1}\varphi(\delta^{-2}\cdot)
\ee 
and letting $\delta\to0$. Equation \eqref{eq:limitequation} is the Euler-Lagrange equation for the functional
\be\label{eq:conformalfunctional}
\cL_0(\varphi)=\frac{1}{2}\int_{\R^2}\langle\cD\varphi,\varphi\rangle\,dx-\frac{1}{4}\int_{\R^2}\vert\varphi\vert^4\,dx\,,
\ee
which then is invariant by translation and scaling. More precisely, both terms in \eqref{eq:conformalfunctional} are individually invariant by scaling, that is, given $\varphi\in \mathring{H}^{1/2}(\R^2,\C^2)$, there holds
\be\label{eq:invariant}
\int_{\R^2}\langle\cD\varphi_\delta,\varphi_\delta\rangle\,dx=\int_{\R^2}\langle\cD\varphi,\varphi\rangle\,dx\,,\quad\int_{\R^2}\vert\varphi_\delta\vert^4\,dx=\int_{\R^2}\vert\varphi\vert^4\,dx\,.
\ee

Moreover, as proved in \cite[Section 4]{isobecritical}, those terms are invariant with respect to a conformal change of metric.
\smallskip

Given $S\in\Z$, solutions to \eqref{eq:limitequation} of the form \eqref{eq:ansatz} have been classified in \cite{borrellifrank}. They are given, up scaling \eqref{eq:scaling} and sign change, by
\be
\Psi(r,\theta)=\begin{pmatrix}
v(r) e^{\I S\theta} \\
\I u(r) e^{\I (S+1)\theta}
\end{pmatrix}\,, 
\ee
with
\be\label{eq:Sbubbles}
\quad u(r)=\sigma\sqrt{2|2S+1|} \frac{r^S}{r^{2S+1}+r^{-(2S+1)}}\,,\quad v(r)=\tau\sqrt{2|2S+1|} \frac{r^{-S-1}}{r^{2S+1}+r^{-(2S+1)}}\,,
\ee
where $\sigma,\tau\in\{-1,1\}$ and $\sigma=\tau$ if $S\geq0$ and $\sigma=-\tau$ if $S<0$.

Those solutions being critical points of \eqref{eq:conformalfunctional}, a straightforward computation gives
\be\label{eq:Senergy}
\beta(S):=\cL_0(\Psi)=\frac{1}{4}\int_{\R^2}\vert\Psi\vert^4\,dx=\vert 2S+1\vert\pi\,.
\ee
As explained in Section \ref{sec:compactness} the above energy value is the threshold for the appearance of the blow-up profiles in the Palais-Smale sequences for the functional $\cL_\omega$, see \eqref{eq:action}.
\smallskip

We add that, more generally, solutions of the form \eqref{eq:ansatz} to the limit equation with the nonlinearity \eqref{eq:generalcubic} have also been characterized in \cite{borrellifrank}. Moreover, ground state solutions to the higher dimensional analogue of \eqref{eq:limitequation} have been recently classified in \cite{borrellimalchiodiwu}.
\section{The dual action }\label{sec:dual}
Following the idea of \cite{isobecritical}, we employ duality techniques, exploiting the convexity of the nonlinear term in \eqref{eq:action}. This allows to study an equivalent problem involving a \emph{dual functional}, whose critical points ( and, more generally, whose Palais-Smale sequences) are in one-to-one correspondence with those of the original functional. In particular, we also exploit the fact that the dual functional has a mountain pass geometry.
\smallskip

Let $$\cDmo:=(\cD+m\sigma_3-\omega).$$
The following isomorphisms hold
\begin{equation}
\cDmo :H^{\frac{1}{2}}(\R^2,\C^2)\longrightarrow H^{-\frac{1}{2}}(\R^2,\C^2)
\end{equation}
and 
\begin{equation}
\cDmo :W^{1,4/3}(\R^2,\C^2)\longrightarrow L^{4/3}(\R^2,\C^2)\,.
\end{equation}

Let $A_{\omega}$ and $B_{\omega}$ be the inverse operators, respectively, that is 
\begin{equation}\label{A}
A_{\omega}:=(\cDmo)^{-1}:H^{-\frac{1}{2}}(\R^2,\C^2)\longrightarrow H^{\frac{1}{2}}(\R^2,\C^2)
\end{equation}
and
\begin{equation}
B_{\omega}:=(\cDmo)^{-1}: L^{4/3}(\R^2,\C^2)\longrightarrow W^{1,4/3}(\R^2,\C^2).
\end{equation}
We denote by
\begin{equation}
i:H^{\frac{1}{2}}(\R^2,\C^2)\longrightarrow L^4(\R^2,\C^2)
\end{equation}
and 
\begin{equation}
j: W^{1,4/3}(\R^2,\C^2)\longrightarrow H^{\frac{1}{2}}(\R^2,\C^2)
\end{equation}
the Sobolev embeddings.

Consider the following sequences of maps 
\begin{equation}
\begin{tikzcd}
K_{\omega}:L^{4/3}(\R^2,\C^2)\arrow[r,"i^{*}"]& H^{-\frac{1}{2}}(\R^2,\C^2)\arrow[r,"A_{\omega}"]& H^{\frac{1}{2}}(\R^2,\C^2)\arrow[r,"i"]& L^{4}(\R^2,\C^2)
\end{tikzcd}
\end{equation}
and 
\begin{equation}
\begin{tikzcd}
L^{4/3}(\R^2,\C^2)\arrow[r,"B_{\omega}"]& W^{1,4/3}(\R^2,\C^2)\arrow[r,"j"]&H^{\frac{1}{2}}(\R^2,\C^2),
\end{tikzcd}
\end{equation}
where $i^{*}:L^{4/3}(\R^2,\C^2)\rightarrow H^{-\frac{1}{2}}(\R^2,\C^2)$ is the adjoint of $i$.
Then we have
\begin{equation}
A_{\omega}\circ i^{*}=j\circ B_{\omega},
\end{equation}
and since $\cD$ is self-adjoint we also have 
\begin{equation}
K^{*}_{\omega}=K_{\omega}\,.
\end{equation}

The functional $\mathcal{L}_{\omega}$ is then defined as
$$\mathcal{L}_{\omega}(\psi)=\frac{1}{2}\langle\cDmo\psi,\psi\rangle_{H^{-\frac{1}{2}}\times H^{\frac{1}{2}}}-\mathcal{H}(i(\psi))$$
for $\psi\in H^{\frac{1}{2}}$, where $\langle\cdot,\cdot\rangle_{H^{-\frac{1}{2}}\times H^{\frac{1}{2}}}$ is the duality pairing between $H^{-\frac{1}{2}}(\R^2,\C^2)$ and $H^{\frac{1}{2}}(\R^2,\C^2)$ and $\mathcal{H}$ is the functional on $L^{4}(\R^2,\C^2)$ defined by
$$\mathcal{H}(\psi)=\frac{1}{4}\int_{\R^2}\vert\psi\vert^{4}\,dx$$

The differential of the functional $\mathcal{L}_{\omega}$ then reads as 
\begin{equation}\label{dualdiff}
d\mathcal{L}_{\omega}(\psi)=\cDmo\psi-i^{*}d\mathcal{H}(i(\psi))\in H^{-\frac{1}{2}}(\R^2,\C^2)\,.
\end{equation}
for $\psi\in H^{\frac{1}{2}}(\R^2,\C^2)$.

The \textit{Legendre transform} (see \cite{criticalhamiltonian}) $\mathcal{H}^{*}$ of $\mathcal{H}$ is the functional on $L^{4/3}(\R^2,\C^2)$ defined by 
\begin{equation}
\begin{split}
\mathcal{H}^{*}(\varphi)&=\max\{\langle\psi,\varphi\rangle_{L^{4}\times L^{4/3}}-\mathcal{H}(\psi)\;:\; \psi\in L^{4}(\R^2,\C^2)\} \\& =\frac{1}{2^{+}}\int_{\R^2}\vert\varphi\vert^{4}\,dx.
\end{split}
\end{equation}
We see that $d\mathcal{H}^{*}$ is the inverse of $d\mathcal{H}$, that is
\begin{equation}\label{inversediff}
d\mathcal{H}\circ d\mathcal{H}^{*}=1_{L^{4/3}},\qquad d\mathcal{H}^{*}\circ d\mathcal{H}=1_{L^{4}}.
\end{equation}
Then the dual functional $\mathcal{L}^{*}_{\omega}$ is defined as 
\begin{equation}
\begin{split}
\mathcal{L}^{*}_{\omega}(\varphi)&=\mathcal{H}^{*}(\varphi)-\frac{1}{2}\langle K_{\omega}\varphi,\varphi\rangle_{L^{4}\times L^{4/3}}\\&=\frac{3}{4}\int_{\R^2}\vert\varphi\vert^{4/3}\,dx-\frac{1}{2}\int_{\R^2}\langle K_{\omega}\varphi,\varphi\rangle\,dx\,,
\end{split}
\end{equation}
for $\varphi\in L^{4/3}(\R^2,\C^2)$. It is not hard to see that $\mathcal{L}^{*}_{\omega}$ is of class $C^{1}$.

A relevant property of the dual functional $\mathcal{L}^{*}_{\omega}$ is that its critical points and Palais-Smale sequences are in one-to-one correspondence with the ones of $\mathcal{L}_{\omega}$.
\begin{lemma}\label{lem:correspondence}
There is a one-to-one correspondence between the critical points of $\mathcal{L}_{\omega}$ in $H^{\frac{1}{2}}(\R^2,\C^2)$ and the critical points of $\mathcal{L}^{*}_{\omega}$ in $L^{4/3}(\R^2,\C^2)$.
\end{lemma}
\begin{proof}
Let $\psi\in H^{\frac{1}{2}}$ be a critical point of $\mathcal{L}_{\omega}$. Then by \eqref{dualdiff}, we have $\cDmo\psi=i^{*}d\mathcal{H}(i(\psi))$. Define $\varphi:=d\mathcal{H}(i(\psi))\in L^{4/3}$, so that $\cDmo\psi=i^{*}\varphi$. This implies that $\psi=A_{\omega}\circ i^{*}(\varphi)$ and
\begin{equation}\label{13}
i(\psi)=i\circ A_{\omega}\circ i^{*}(\varphi)=K_{\omega}.
\end{equation}
On the other hand, by~\eqref{inversediff} we have
\begin{equation}\label{14}
i(\psi)=d\mathcal{H}^{*}(\varphi).
\end{equation}
Combining \eqref{13} and \eqref{14} we obtain
$$d\mathcal{L}^{*}_{\omega}(\varphi)=d\mathcal{H}^{*}(\varphi)-K_{\omega}(\varphi)=0. $$
and then $\varphi$ is a critical point of $\mathcal{L}^{*}_{\omega}$.

Conversely, suppose $\varphi\in L^{4/3}$ is a critical point of $\mathcal{L}^{*}_{\omega}$ and define $\psi=A_{\omega}\circ i^{*}(\varphi)\in H^{\frac{1}{2}}$. Since $\varphi$ is a critical point, we have $d\mathcal{H}^{*}(\varphi)-K_{\omega}(\varphi)=0$. This and~\eqref{inversediff} imply that 
\begin{equation}
\varphi=d\mathcal{H}\circ K_{\omega}(\varphi)=d\mathcal{H}\circ i\circ A_{\omega}\circ i^{*}(\varphi)=d\mathcal{H}(i(\psi)).
\end{equation}
Then we have $i^{*}(\varphi)=i^{*}\circ d\mathcal{H}(i(\psi))$ and $d\mathcal{L}_{\omega}(\psi)=\cDmo\psi=i^{*}\circ d\mathcal{H}(i(\psi))$, that is $\psi$ is a critical point of $\mathcal{L}_{\omega}$. This concludes the proof.
\end{proof}
Moreover, there also exists a one-to-one correspondence between Palais-Smale sequences for $\mathcal{L}_{\omega}$ and $\mathcal{L}^{*}_{\omega}$. We refer the reader to \cite[Section 3 ]{isobecritical}

 \section{Compactness analysis}\label{sec:compactness}
In this section we analyze the compactness properties of the functional $\cL_\omega$, restricted to the subspace $E_S$. As explained in the previous section, the same results hold for the dual functional $\cL^*_\omega$. The arguments employed rely on Concentration-Compactness theory, for which the reader can refer to  \cite{CCcompactI,CClimitI} for a general exposition and to \cite{Palatucci-Pisante} for the case of fractional Sobolev spaces needed in the sequel.
 \medskip
 
Let $(\psi_n)_{n\in\N}\subseteq H^{1/2}(\R^2,\C^2)$ be a Palais-Smale sequence for $\cL_\omega$ at level $c>0$, that is
\be\label{eq:PS}
\cL_\omega(\psi_n)\rightarrow c\,,\quad \dd{\cL_\omega}(\psi_n)\xrightarrow{H^{-1/2}}0\,, 
\ee
as $n\rightarrow\infty$. It is not hard to see that $\psi_n$ is bounded. 
\begin{lemma}\label{lem:boundedPS}
Any Palais-Smale for $\cL_\omega$ is bounded.
\end{lemma}
\begin{proof}
There holds
\[
(\cD+m\sigma_3-\omega)\psi_n=\vert\psi_n\vert^2\psi_n+o(1)\,,\qquad\mbox{in $H^{-\frac{1}{2}}(\R^2,\C^2)$}\,,
\]
and thus
\be\label{eq:psi_n}
\psi_n=(\cD+m\sigma_3-\omega)^{-1}(\vert\psi_n\vert^2\psi_n)+o(1)\,,\qquad\mbox{in $H^{\frac{1}{2}}(\R^2,\C^2)$}\,.
\ee
From this we get
\[
\Vert\psi_n\Vert_{H^{1/2}}\lesssim\Vert \vert\psi_n\vert^2\psi_n\Vert_{H^{-1/2}}+o(1)\,,
\]
and by the Sobolev embedding $H^{\frac{1}{2}}(\R^2,\C^2)\hookrightarrow L^4(\R^2,\C^2)$, there holds $L^{\frac{4}{3}}(\R^2,\C^2)=(L^4(\R^2,\C^2))^*\hookrightarrow H^{-\frac{1}{2}}(\R^2,\C^2)$, and then
\[
\Vert \vert\psi_n\vert^2\psi_n\Vert_{H^{-1/2}}\lesssim\Vert \vert\psi_n\vert^2\psi_n\Vert_{L^{4/3}}=\Vert\psi_n\Vert^3_{L^4}\,.
\]
Moreover, by \eqref{eq:PS} we deduce that
\[
\begin{split}
\frac{1}{4}\int_{\R^2}\vert\psi_n\vert^4\,dx&=\cL_\omega(\psi_n)-\frac{1}{2}\langle d\cL_\omega(\psi_n),\psi_n\rangle_{H^{-1/2}\times H^{1/2}} \\
&\leq C+\Vert\psi_n\Vert_{H^{1/2}}\,,
\end{split}
\]
and then, combining the above observations we find
\[
\Vert\psi_n\Vert_{H^{1/2}}\leq C\Vert\psi_n\Vert^3_{L^4}\leq C(1+\Vert\psi_n\Vert_{H^{1/2}})^{3/4}\,,
\]
and the claim follows.
\end{proof}

Then, up to subsequences, we have
\be\label{eq:weakH^1/2}
\psi_n\rightharpoonup\psi_\infty\,,\qquad\mbox{weakly in $H^{1/2}(\R^2,\C^2)$}\,,
\ee
and
\begin{align}\label{eq:L^p-convergence}
 &\psi_n\to\psi_\infty\,,\qquad\mbox{strongly in $L^{p}_{loc}(\R^2,\C^2)$, for $2\leq p<4$}\,, \\
&\psi_n\rightharpoonup\psi_\infty\,,\qquad\mbox{weakly in $L^{4}(\R^2,\C^2)$}\,,
\end{align}
as $n\to\infty$. 

The strong $H^{1/2}$-convergence of Palais-Smale sequences is a priori prevented by the invariance by translation of the functional, and by the presence of a critical nonlinearity. More precisely, we need to prove \emph{strong} convergence in the $L^4$ norm.
\begin{proposition}\label{prop:strong}
Assume 
\be\label{eq:L4-strong}
\psi_n\to\psi_\infty\,,\qquad\mbox{strongly in $L^{4}(\R^2,\C^2)$}\,,
\ee
then
\be\label{eq:Sobolev-strong}
\psi_n\to\psi_\infty\,,\qquad\mbox{strongly in $H^{1/2}(\R^2,\C^2)$}\,.
\ee
\end{proposition}
\begin{proof}
We claim that 
\be\label{eq:square}
\vert\psi_n\vert^2\psi_n\to\vert\psi_\infty\vert^2\psi_\infty\,,\qquad \mbox{strongly in $L^{4/3}(\R^2,\C^2)$}\,.
\ee 

Preliminarily, observe that, up to subsequences, 
\be\label{eq:L^2}
\vert\psi_n\vert^2\to\vert\psi_\infty\vert^2\,,\qquad  \mbox{strongly in $L^2(\R^2,\C^4)$.}
\ee
 Indeed, the sequence is bounded in $L^2$ as $\Vert\vert\psi_n\vert^2\Vert_{L^2}=\Vert\psi_n\Vert^2_{L^4}\leq C$, uniformly in $n\in\N$. Then $\vert\psi_n\vert^2\rightharpoonup \vert\psi_\infty\vert^2$, weakly in $L^2$. Moreover, $\Vert\vert\psi_n\vert^2\Vert_{L^2}=\Vert\psi_n\Vert^2_{L^4}\to\Vert\psi_\infty\Vert^2_{L^4}=\Vert\vert\psi_\infty\vert^2\Vert_{L^2}$, by \eqref{eq:L4-strong}, and the $L^2$ strong convergence follows.
 
 There holds
 \[
 \Vert \vert\psi_n\vert^2\psi_n-\vert\psi_\infty\vert^2\psi_\infty\Vert_{L^{4/3}}\leq  \Vert \vert\psi_n\vert^2\psi_n-\vert\psi_n\vert^2\psi_\infty\Vert_{L^{4/3}}+ \Vert \vert\psi_n\vert^2\psi_\infty-\vert\psi_\infty\vert^2\psi_\infty\Vert_{L^{4/3}}
 \]

The H\"older inequality and \eqref{eq:L4-strong} give
\be
\begin{split}
 \Vert \vert\psi_n\vert^2\psi_n-\vert\psi_n\vert^2\psi_\infty\Vert^{4/3}_{L^{4/3}}&=\int_{\R^2}\vert\vert\psi_n\vert^2\psi_n-\vert\psi_n\vert^2\psi_\infty \vert^{\frac{4}{3}}\,dx \\
  &\leq\left(\int_{\R^2}\vert\psi_n\vert^4 \,dx\right)^{\frac{2}{3}}\left(\int_{\R^2}\vert\psi_n-\psi_\infty\vert^4 \,dx\right)^{\frac{1}{3}}\\
  &\leq C\left(\int_{\R^2}\vert\psi_n-\psi_\infty\vert^4 \,dx\right)^{\frac{1}{3}}=o(1)\,.
\end{split}
\ee
Similarly, by H\"older and \eqref{eq:L^2} we get
\be
\begin{split}
 \Vert \vert\psi_n\vert^2\psi_\infty-\vert\psi_\infty\vert^2\psi_\infty\Vert_{L^{4/3}}&=\int_{\R^2}\vert\vert\psi_n\vert^2\psi_\infty-\vert\psi_\infty\vert^2\psi_\infty \vert^{\frac{4}{3}}\,dx\\
  & \leq\left( \int_{\R^2}\vert\vert\psi_n\vert^2-\vert\psi_\infty\vert^2 \vert^2\,dx\right)^{\frac{2}{3}}\left( \int_{\R^2}\vert\psi_\infty \vert^4\,dx\right)^{\frac{1}{3}}=o(1)\,.
\end{split}
\ee
This proves \eqref{eq:square} and thus \eqref{eq:Sobolev-strong} follows, by \eqref{eq:psi_n}.
\end{proof}
\begin{proposition}\label{prop:limitsolution}
The spinor $\psi_\infty\in H^{1/2}(\R^2,\C^2)$ is a weak solution to \eqref{criticaldirac}.
\end{proposition}
\begin{proof}
Take $\varphi\in C^\infty_c(\R^2,\C^2)$. Since $\psi_n$ is a Palais-Smale sequence for $\cL_\omega$, we get
\be\label{eq:PStested}
\begin{split}
o(1)&=\langle d\cL_\omega(\psi_n),\varphi\rangle_{H^{-\frac{1}{2}}\times H^{\frac{1}{2}}} \\
 &=\int_{\R^2}\langle\psi_n,\cD\varphi\rangle\,dx+\int_{\R^2}\langle(m\sigma_3-\omega)\psi_n,\varphi\rangle\,dx-\int_{\R^2}\vert\psi_n\vert^2\langle\psi_n,\varphi\rangle\,dx\,.
\end{split}
\ee
By \eqref{eq:L^p-convergence} we easily get that the first, second and third terms on the right-hand side of \eqref{eq:PStested} converge to $\int_{\R^2}\langle\psi_\infty,\cD\varphi\rangle\,dx,\int_{\R^2}\langle(m\sigma_3-\omega)\psi_\infty,\varphi\rangle\,dx,\int_{\R^2}\vert\psi_\infty\vert^2\langle\psi_\infty,\varphi\rangle\,dx$, respectively, as $n\to\infty$.
\end{proof}

Choose $S\in\Z, S\neq0$ and consider a Palais-Smale sequence $(\psi_n)_n$ for $\cL_\omega$ restricted to the subspace $E_S$ of symmetric spinors of the form \eqref{eq:ansatz}. This ansatz breaks the invariance by translations and allows to prove that the only possible concentration point is given by the origin, as shown in the following
\begin{lemma}
There exists $\nu\geq0$ such that 
\be\label{eq:measure-limit}
\vert\psi_n\vert^4\,dx\overset{\ast}{\rightharpoonup}\vert\psi_\infty\vert^4\,dx+\nu\delta_0\,,\qquad\mbox{in $\mathcal{M}(\R^2)$}\,,
\ee
where $\delta_0$ is the delta measure concentrated at the origin.
\end{lemma}
\begin{proof}
By \cite[Theorem 5]{Palatucci-Pisante} there exists a (at most countable) set of distinct points $x_j\in\R^2$ and of numbers $\nu_j\geq0$, $j\in J$, such that
\be\label{eq:measure-concentration}
\vert\psi_n\vert^4\,dx\overset{\ast}{\rightharpoonup}\vert\psi_\infty\vert^4\,dx+\sum_{j\in J}\nu_j\delta_{x_j}\,,\qquad\mbox{in $\mathcal{M}(\R^2)$}\,,
\ee
where the $\delta_{x_j}$ are delta measures at $x_j$. We claim that $\vert J\vert=1$ and the only concentration point is $x=0$. Indeed, observe that spinors of the form \eqref{eq:ansatz} are invariant by the following $\mathbb{S}^1$-action. Given $\theta\in[0,2\pi)$, there holds
\be\label{eq:S1-action}
\begin{split}
\mathcal{R}_\theta(\psi(r,\varphi)):=\begin{pmatrix} e^{-\I S\theta} & \\ 0 & e^{-\I(S+1)\theta}\end{pmatrix}\begin{pmatrix}
v(r) e^{\I S(\varphi+\theta)} \\
\I u(r) e^{\I (S+1)(\varphi+\theta)}
\end{pmatrix} =\psi(r,\varphi)\,.
\end{split}
\ee
Then, given a point $x_j\neq0$ in \eqref{eq:measure-concentration}, by \eqref{eq:S1-action} $R_\theta x_j\neq x_j$ is also a concentration point, $R_\theta$ being the counterclockwise rotation of angle $\theta$ in $\R^2$. But this contradicts the fact that $J$ must be at most countable, and \eqref{eq:measure-limit} follows.
\end{proof}
\begin{remark}\label{rmk:strongL4}
Observe that if $\nu=0$ in \eqref{eq:measure-limit}, by reflexivity we get \emph{strong} convergence $\psi_n\to\psi_\infty$ in $L^4(\R^2,\C^2)$, as in that case we have \emph{weak} $L^4$ convergence and convergence of the norm, i.e. $\Vert\psi_n\Vert_{L^4}\to\Vert\psi_\infty\Vert_{L^4}$, as $n\to\infty$. 
\end{remark}
However, again by \eqref{eq:measure-limit}, for any $\eps>0$ there holds
\be\label{eq:strong-outside}
\psi_n\to\psi_\infty\,,\qquad \mbox{strongly in $L^{4}(\R^2\setminus B_\eps,\C^2)$}\,.
\ee
Combining this fact with \eqref{eq:L^p-convergence} we can prove strong $L^2$-convergence.
\begin{proposition}\label{eq:L2strong}
There holds
\be\label{eq:strongL2}
\psi_n\to\psi_\infty\,,\qquad \mbox{strongly in $L^{2}(\R^2,\C^2)$}\,,
\ee
as $n\to\infty$.
\end{proposition}
\begin{proof}
By \eqref{eq:L^p-convergence}, we only need to prove that strong convergence holds in $L^2(\R^2\setminus B_R,\C^2)$, for any $R>0$, exploiting \eqref{eq:strong-outside}.

To this aim, recall that being a Palais-Smale sequence $\psi_n$ verifies
\be\label{eq:PS2}
\cDmo\psi_n=\vert\psi_n\vert^2\psi_n+o(1)\,,\qquad\mbox{in $H^{-\frac{1}{2}}(\R^2,\C^2)$}\,.
\ee
Fix $R>0$ and take a smooth function $f\in C^\infty(\R^2)$ with $\supp f\subseteq \R^2\setminus B_R$, $0\leq f\leq1$ and $f\equiv1$ on $\R^2\setminus B_{2R}$. Observe that 
\[
f\cDmo=\cDmo f+[f,\cDmo]\,,
\]
where the commutator $[f,\cDmo]=-\I\sigma\cdot\nabla f$ is supported on $B_{2R}\setminus B_R$. Then by \eqref{eq:PS2} we get 
\[
\cDmo(f\psi_n)=-[f,\cDmo]\psi_n+f\vert\psi_n\vert^2\psi_n+o(1)\,,\qquad\mbox{in $H^{-\frac{1}{2}}(\R^2,\C^2)$}\,.
\]
Similarly, since $\psi_\infty$ is a weak solution to \eqref{criticaldirac} there holds
\[
\cDmo(f\psi_\infty)=-[f,\cDmo]\psi_\infty+f\vert\psi_\infty\vert^2\psi_\infty\,,\qquad\mbox{in $H^{-\frac{1}{2}}(\R^2,\C^2)$}\,.
\]
Arguing as for \eqref{eq:square}, one gets $f\vert\psi_n\vert^2\psi_n\to f\vert\psi_\infty\vert^2\psi_\infty$ strongly in $L^{\frac{4}{3}}$, as $n\to\infty$. As remarked, the commutator $[f,\cDmo]$ has compact support and so by \eqref{eq:L^p-convergence} we also get $[f,\cDmo]\psi_n\to [f,\cDmo]\psi_\infty$ strongly in $L^{\frac{4}{3}}$, as $n\to\infty$. Then, inverting $\cDmo$ in the above equations we finally get $f\psi_n\to f\psi_\infty$ strongly in $L^{2}$ and the claim follows.
\end{proof}
The result in \eqref{eq:measure-limit} can be rephrased in terms of a profile decomposition (see \cite[Theorem 4]{Palatucci-Pisante}). If $\nu>0$ in \eqref{eq:measure-limit}, then there holds
\be\label{eq:profile-decomp}
\psi_n=\psi_\infty+\sqrt{\lambda_n}\Psi\left(\lambda_n(\cdot-x_n)\right)+r_n\,,
\ee
where $x_n\in\R^2$, $x_n\to0$ and $\lambda_n\to\infty$. Here $\Psi$ is a \emph{bubble} as in \eqref{eq:Sbubbles} and $r_n=o(1)$ in $\mathring{H}^{\frac{1}{2}}(\R^2,\C^2)$. The rescaled profile $\Psi$ is peaking at the origin, preventing strong $L^4$ convergence. The next result shows that it also carries part of the `energy' of the Palais-Smale sequence.
\begin{lemma}\label{lem:energy-split}
Set $\varphi_n:=\psi_n-\psi_\infty$. There holds
\be\label{eq:energy-split}
\cL_\omega(\psi_n)=\cL_\omega(\psi_\infty)+\cL_0(\varphi_n)+o(1)\,,\qquad\mbox{as $n\to\infty$,}
\ee
where $\cL_0(\varphi)=\frac{1}{2}\int_{\R^2}\langle\cD\varphi,\varphi\rangle\,dx-\frac{1}{4}\int_{\R^2}\vert\varphi\vert^4\,dx$, see Section \ref{sec:preliminaries}.
\end{lemma} 
\begin{proof}
Recalling that $\psi_n=\varphi_n+\psi_\infty$, we have
\be\label{eq:partialsplit}
\begin{split}
\cL_\omega(\psi_n)=\frac{1}{2}\int_{\R^2}\langle\cD(\varphi_n+\psi_\infty),\varphi_n+\psi_\infty\rangle\,dx&+\frac{1}{2}\int_{\R^2}\langle(m\sigma_3-\omega)(\varphi_n+\psi_\infty),\varphi_n+\psi_\infty \rangle\,dx \\
 &+\frac{1}{4}\int_{\R^2}\vert\varphi_n+\psi_\infty\vert^4\,dx\,,
\end{split}
\ee
and then by \eqref{eq:weakH^1/2}, \eqref{eq:L^p-convergence} and \eqref{eq:strongL2} we find
\be
\begin{split}
\cL_\omega(\psi_n)=\int_{\R^2}\langle\cD\varphi_n,\varphi_n\rangle\,+\langle\cD\psi_\infty,\psi_\infty\rangle\,dx&+\frac{1}{2}\int_{\R^2}\langle(m\sigma_3-\omega)\psi_\infty,\psi_\infty\rangle\,dx\\
 &-\frac{1}{4}\int_{\R^2}\vert\varphi_n+\psi_\infty\vert^4\,dx+o(1)\,,
\end{split}
\ee
as $n\to\infty$. Moreover, by the Brezis-Lieb lemma \cite{Brezis-Lieb} there holds
\[
\int_{\R^2}\vert\varphi_n+\psi_\infty\vert^4\,dx=\int_{\R^2}\vert\varphi_n\vert^4\,dx+\int_{\R^2}\vert\psi_\infty\vert^4\,dx+o(1)\,,\qquad\mbox{as $n\to\infty$,}
\]
and combining it with \eqref{eq:partialsplit} we get \eqref{eq:energy-split}.
\end{proof}
We are now in a position to give the following compactness result.
\begin{lemma}\label{lem:compactness}
Let $(\psi_n)_n\in E_S$ be a Palais-Smale sequence for $\cL_{\omega}$ at level $c\geq0$, i.e. $\lim_{n\to\infty}\cL_{\omega}(\psi_n)=c$. If 
\be\label{eq:threshold}
c<\beta(S):=(2S+1)\pi\,,
\ee
then $(\psi_n)_n$ is compact in $H^{\frac{1}{2}}(\R^2,\C^2)$.
\end{lemma}
\begin{proof}
We argue by contradiction, assuming \eqref{eq:threshold} holds and $\psi_n$ is not compact in $H^{\frac{1}{2}}$. Then $\psi_n\not\to\psi_\infty$, where $\psi_\infty$ is its \emph{weak} limit (see \eqref{eq:weakH^1/2}). Then by Prop. \ref{prop:strong}, $\psi_n\not\to\psi_\infty$ in strong sense in $L^4$, so that $\nu>0$ in \eqref{eq:measure-limit}.

Combining the profile decomposition \eqref{eq:profile-decomp} and \eqref{eq:energy-split} we get 
\[
\cL_\omega(\psi_n)=\cL_\omega(\psi_\infty)+\cL_0(\sqrt{\lambda_n}\Psi(\lambda_n(\cdot-x_n))+r_n)+o(1)\,.
\]
Since $\psi_\infty$ is a weak solution to \eqref{criticaldirac}, there holds $\cL_\omega(\psi_\infty)=\frac{1}{4}\int_{\R^2}\vert\psi_\infty\vert^4\,dx\geq0$. Recall that the two terms in $\cL_0$ are invariant by translations and scaling (see \eqref{eq:invariant}), and then $\cL_0(\sqrt{\lambda_n}\Psi(\lambda_n(\cdot-x_n))=\cL_0(\Psi)$. We thus find
\be
\begin{split}
\cL_0(\sqrt{\lambda_n}\Psi(\lambda_n(\cdot-x_n))+r_n)&=\frac{1}{2}\int_{\R^2}\langle\cD\Psi,\Psi\rangle\,dx+\frac{1}{2}\int_{\R^2}\langle\cD r_n,r_n\rangle\,dx \\
& -\frac{1}{4}\int_{\R^2}\vert\sqrt{\lambda_n}\Psi(\lambda_n(\cdot-x_n))+r_n\vert^4\,dx+o(1)\,,
\end{split}
\ee
using the fact that $r_n=o(1)$ in $H^{\frac{1}{2}}$. Moreover, by the Brezis-Lieb lemma \cite{Brezis-Lieb} we get
\be
\begin{split}
\frac{1}{4}\int_{\R^2}\vert\sqrt{\lambda_n}\Psi(\lambda_n(\cdot-x_n))+r_n\vert^4\,dx&=\frac{1}{4}\int_{\R^2}\vert\Psi\vert^4\,dx+\frac{1}{4}\int_{\R^2}\vert r_n\vert^4\,dx+o(1) \\
&=\frac{1}{4}\int_{\R^2}\vert\Psi\vert^4\,dx+o(1)\,.
\end{split}
\ee
Then, by the above observations and using \eqref{eq:Senergy} we conclude that
\[
\liminf_{n\to\infty}\cL_\omega(\psi_n)\geq\lim_{n\to\infty} [\cL_0(\Psi)+\cL_0(r_n)+o(1)]=\beta(S)\,,
\]
contradicting the assumption \eqref{eq:threshold}.
\end{proof}

\section{Proof of Theorem \ref{thm:main}}\label{sec:proof}
In this section we give the proof of the main Theorem \eqref{thm:main}. For the convenience of the reader we divide the proof into different steps, so that it will be achieved combining Proposition \ref{prop:existence}, \ref{prop:regularity} and \ref{prop:expdecay}.

\subsection{Existence of solutions}\label{sec:existence}
The results of Section \ref{sec:dual} show that finding a critical point of $\mathcal{L}_{\omega}$ is equivalent to the same problem for the dual functional $\mathcal{L}^{*}_{\omega}$. This allows us to exploit the fact that the latter possesses a mountain pass geometry (see e.g. \cite{struwevariational}). 
\begin{lemma}
There exists $\rho>0$ such that 
\be\label{eq:positivesmallball}
\inf\{\mathcal{L}^{*}_{\omega}(\varphi)\;:\;\varphi\in L^{4/3}(\R^2,\C^2),\;\Vert\varphi\Vert_{L^{4/3}}=\rho\}>0\,.
\ee
Moreover, given $\varphi\in L^{4/3}(\R^2,\C^2)$ such that $\int_{\R^2}\langle \varphi,A_{\omega}\varphi\rangle>0$, there holds 
\be\label{eq:negativealongray}
\lim_{t\rightarrow+\infty}\mathcal{L}^{*}_{\omega}(t\varphi)=-\infty\,.
\ee
\end{lemma}
\begin{proof}
Recall that $\cL^*_\omega(\varphi)=\frac{3}{4}\int_{\R^2}\vert\varphi\vert^{\frac{4}{3}}\,dx-\frac{1}{2}\int_{\R^2}\langle\varphi,A_\omega\varphi\rangle\,dx$ so that by the Sobolev embedding
\[
\begin{split}
\left\vert\int_{\R^2}\langle\varphi,A_\omega\varphi\rangle\,dx\right\vert&\leq\Vert\varphi\Vert_{L^{\frac{4}{3}}}\Vert A_\omega\varphi\Vert_{L^4}\\
&\leq\Vert\varphi\Vert_{L^{\frac{4}{3}}}\Vert A_\omega\varphi\Vert_{W^{1,4/3}}\leq\Vert\varphi\Vert^2_{L^{\frac{4}{3}}}
\end{split}
\]
so that \eqref{eq:positivesmallball} follows for $\rho:=\Vert\varphi\Vert_{L^{4}{3}}$ small. Moreover, \eqref{eq:negativealongray} easily follows since for a fixed $\varphi\in L^{4/3}(\R^2,\C^2)$ with $\int_{\R^2}\langle \varphi,A_{\omega}\varphi\rangle>0$, there holds 
\[
\cL^*_\omega(t\varphi)=\frac{3}{4}t^{\frac{4}{3}}\int_{\R^2}\vert\varphi\vert^{\frac{4}{3}}\,dx-\frac{1}{2}t^2\int_{\R^2}\langle\varphi,A_\omega\varphi\rangle\,dx\,,\qquad t\geq0\,.
\]
\end{proof}

The mountain pass level for $\mathcal{L}^{*}_{\omega}$ is defined as 
\be\label{eq:mplevel}
c_{\omega}:=\inf \bigg\{\max_{t\geq0}\mathcal{L}^{*}_{\omega}(t\varphi)\;:\;\varphi\in L^{4/3}(\R^2,\C^2), \int_{\R^2}\langle\varphi, A_{\omega}\varphi\rangle\,dx>0\;\bigg\}. 
\ee
It can be easily shown that
$$c_{\omega}:=\inf\bigg\{\frac{1}{4}\frac{(\int_{\R^2}\vert\varphi\vert^{4/3}\,dx)^{3}}{(\int_{\R^2}\Re\langle\varphi,A_{\omega}\varphi\rangle\,dx)^{2}}\;:\;\varphi\in L^{4/2}(\R^2,\C^2), \int_{\R^2}\langle\varphi, A_{\omega}\varphi\rangle\,dx>0\bigg\}.  $$

Set
\be\label{eq:quotient}
J(\phi):=\frac{1}{4}\frac{(\int_{\R^2}\vert\varphi\vert^{4/3}\,dx)^{3}}{(\int_{\R^2}\Re\langle\varphi,A_{\omega}\varphi\rangle\,dx)^{2}}\,.
\ee
According to the compactness analysis described in Section \ref{sec:compactness}, we need to find a suitable test spinor $\widetilde{\varphi}\in L^{4/3}(\R^2,\C^2)$ such that 
\[
J(\widetilde{\varphi})<\beta\,,
\]
where $\beta$ is the lower bound for the energy of the bubbles. This allows to recover compactness of Palais-Smale sequences and to get the existence of a critical point of $\cL_\omega$.

\medskip

In what follows we fix $S\in\Z\setminus\{0\}$ and consider spinors of the form \eqref{eq:ansatz}, accordingly. We assume $S>0$, as the case $S<0$ follows by the same arguments.
 In this case the bubbles are given by \eqref{eq:Sbubbles}, so that the threshold energy becomes $\beta(S)=(2S+1)\pi$, see \eqref{eq:Senergy}.

\medskip

\begin{lemma}\label{lem:inequality}
There exists a spinor $\widetilde{\varphi}\in L^{4/3}(\R^2,\C^4)$ of the form \eqref{eq:ansatz} such that
\be\label{eq:inequality}
J(\widetilde{\varphi})<\beta(S)=(2S+1)\pi
\ee
\end{lemma}
\begin{proof}
Consider the bubble $\Psi$ in \eqref{eq:Sbubbles}. Given $\eps>0$ define
\be\label{eq:psieps}
\psi_\eps(x):=\theta(x)\Psi(x/\eps)\,,\qquad x\in\R^2\,,
\ee
where $\theta\in C^\infty_c(\R^2)$, $0\leq\theta\leq 1$, is a cutoff function supported in $B_2(0)$, with $\eta\equiv1$ on $B_1(0)$.  
Define
\be\label{eq:testspinor}
\varphi_\eps(x):=\cD\psi_\eps(x)\,,\qquad x\in\R^2\,.
\ee
Our aim is to show that we can choose $\widetilde{\varphi}=\varphi_\eps$, for suitable $0<\eps\ll1$.

\emph{Step 1: estimate of the numerator.}  Recalling that $\cD\Psi=\vert\Psi\vert^2\Psi$, we have
\be\label{eq:phisquared}
\begin{split}
\vert\varphi_\eps\vert^2&=\vert\eps^{-1}\theta\cD\Psi(x/\eps)-\I(\sigma\cdot\nabla\theta)\Psi(x/\eps)\vert^2 \\
&= \eps^{-2}\theta^2\vert\Psi(x/\eps)\vert^6+\vert\nabla\theta\vert^2\vert\Psi(x/\eps)\vert^2 \\
&+2\eps^{-1}\theta\vert\Psi(x/\eps)\vert^2\underbrace{\Re\langle\Psi(x/\eps),(-\I\sigma\cdot\nabla\theta)\Psi(x/\eps) \rangle}_{=0}\,,
\end{split}
\ee
where the last term vanishes as the matrix $-\I\sigma\cdot\nabla\theta$ is skew-hermitian. 

On $B_2(0)$ we have
\[
\begin{split}
\vert\varphi_\eps\vert^2 &\leq \eps^{-2}\vert\Psi(x/\eps)\vert^6+\vert\nabla\theta\vert^2\vert\Psi(x/\eps)\vert^2 \\
&\leq\eps^{-2}\vert\Psi(x/\eps)\vert^6\left(1+\eps^2\vert\nabla\theta\vert^2\vert\Psi(x/\eps)\vert^{-4} \right)\,.
\end{split}
\]
The elementay inequality $(1+t)^{\frac{2}{3}}\leq1+t^{\frac{2}{3}}$, $t\geq0$, gives on $B_2(0)$
\[
\begin{split}
\vert\varphi_\eps\vert^{\frac{4}{3}}&\leq\eps^{-\frac{4}{3}}\vert\Psi(x/\eps)\vert^4\left(1+\eps^{\frac{4}{3}}\vert\nabla\theta\vert^2\vert\Psi(x/\eps)\vert^{-\frac{8}{3}} \right)\\
&=\eps^{-\frac{4}{3}}\vert\Psi(x/\eps)\vert^4+\vert\nabla\theta\vert^2\vert\Psi(x/\eps)\vert^{\frac{4}{3}}
\end{split}
\]

Observing that $\varphi_\eps$ is supported on $B_2(0)$ and $\supp \nabla\theta\subseteq B_2(0)\setminus B_1(0)$ we find
\[
\begin{split}
\int_{\R^2}\vert\varphi_\eps\vert^{\frac{4}{3}} \,dx&\leq\eps^{-\frac{4}{3}}\int_{B_2(0)}\vert\Psi(x/\eps)\vert^4 \,dx+C\int_{B_2\setminus B_1}\vert\Psi(x/\eps)\vert^{\frac{4}{3}} \\
&=\eps^{\frac{2}{3}}\int_{B_\frac{2}{\eps}}\vert\Psi\vert^4\,dx+C\eps^{2}\underbrace{\int_{B_{\frac{2}{\eps}}\setminus B_{\frac{1}{\eps}}}\vert\Psi\vert^{\frac{4}{3}}\,dx}_{=o_\eps(1)}\,,
\end{split}
\]
as $\Psi\in L^{\frac{4}{3}}$ for $S\neq0$, see \eqref{eq:Sbubbles}. Then there holds
\[
\int_{\R^2}\vert\varphi_\eps\vert^{\frac{4}{3}} \,dx\leq\eps^{\frac{2}{3}}\int_{B_\frac{2}{\eps}}\vert\Psi\vert^4\,dx+o_\eps(1)=4(2S+1)\pi \eps^{\frac{2}{3}}+o(\eps^{2})\,.
\]
by \eqref{eq:Senergy}, so that
\be\label{eq:num-bound}
\left( \int_{\R^2}\vert\varphi_\eps\vert^{\frac{4}{3}} \,dx\right)^3\leq4^3(2S+1)^3\pi^3 \eps^{2}+o(\eps^{\frac{10}{3}})\,.
\ee
\emph{Step 1: estimate of the denominator.} 
Recall that $A_\omega=(\cD+m\sigma_3-\omega)^{-1}$ and let $\eta_\eps$ be defined setting
\be\label{eq:def-eta}
A_\omega\varphi_\eps=\psi_\eps+\eta_\eps\,,
\ee
so that 
\be\label{eq:eta}
\eta_\eps=A_\omega(\omega-m\sigma_3)\psi_\eps\,.
\ee
There holds
\be\label{eq:denominator}
\int_{\R^2}\Re\langle\varphi_\eps,A_\omega\varphi_\eps\rangle\,dx=\int_{\R^2}\Re\langle\varphi_\eps,\psi_\eps\rangle\,dx+\int_{\R^2}\Re\langle\varphi_\eps,\eta_\eps\rangle\,dx\,.
\ee
By \eqref{eq:testspinor}, we have
\[
\begin{split}
\int_{\R^2}\Re\langle\varphi_\eps,\psi_\eps\rangle\,dx&=\int_{\R^2}\theta^2\Re\langle\eps^{-1}\cD\Psi(x/\eps),\Psi(x/\eps)\rangle\,dx\\
&+\int_{\R^2}\theta\underbrace{\Re\langle-\I(\sigma\cdot\nabla\theta)\Psi(x/\eps),\Psi(x/\eps)\rangle}_{=0}\,dx\,,
\end{split}
\]
the matrix $-\I(\sigma\cdot\nabla\theta)$ being skew-hermitian. Then we find, by the definition of $\theta$,
\[
\int_{\R^2}\Re\langle\varphi_\eps,\psi_\eps\rangle\,dx=\int_{\R^2}\eps^{-1}\theta^2\vert\Psi(x/\eps)\vert^4\,dx\geq\int_{B_1}\eps^{-1}\vert\Psi(x/\eps)\vert^4\,dx\,.
\]
Observe that
\[
\eps^{-1}\int_{B_1}\vert\Psi(x/\eps)\vert^4\,dx=\eps\int_{B_{\frac{1}{\eps}}}\vert\Psi\vert^4\,dx=\eps\int_{\R^2}\vert\Psi\vert^4\,dx - \eps\int_{\R^2\setminus B_{\frac{1}{\eps}}}\vert\Psi\vert^4\,dx\,
\]
By \eqref{eq:Sbubbles} we deduce that $\vert\Psi\vert^4\sim r^{-4(S+1)}$ as $r\to\infty$, so that passing to polar coordinates we find
\[
\int_{\R^2\setminus B_{\frac{1}{\eps}}}\vert\Psi\vert^4\,dx\lesssim \int^\infty_{\frac{1}{\eps}}r^{-4S-3} \,dr=\mathcal{O}(\eps^{4S+2})\,,
\]
and then we get
\be\label{eq:first-term}
\int_{\R^2}\Re\langle\varphi_\eps,\psi_\eps\rangle\,dx\geq\eps\int_{\R^2}\vert\Psi(x/\eps)\vert^4\,dx+\mathcal{O}(\eps^{4S+3})=4(2S+1)\eps\pi+\mathcal{O}(\eps^{4S+3})\,.
\ee
\smallskip

We now turn to the second term on the right-hand side of \eqref{eq:denominator}. By \eqref{eq:eta} and \eqref{eq:def-eta} we find
\[
\begin{split}
\int_{\R^2}\Re\langle\varphi_\eps,\eta_\eps\rangle\,dx&=\int_{\R^2}\Re\langle A_\omega\varphi_\eps,(\omega-m\sigma_3)\psi_\eps\rangle\,dx\\
&=\int_{\R^2}\langle\psi_\eps,(\omega-m\sigma_3)\psi_\eps\rangle\rangle\,dx+\int_{\R^2}\langle\eta_\eps,(\omega-m\sigma_3)\psi_\eps\rangle\,dx\,.
\end{split}
\]
Observe that $\supp\psi_\eps\subseteq B_2$, so that elliptic estimates for the Dirac operator (see \eqref{eq:eta}) and the Sobolev embedding give
\[
\begin{split}
\left\vert\int_{\R^2}\langle\eta_\eps,(\omega-m\sigma_3)\psi_\eps\rangle\,dx\right\vert&=\left\vert\int_{B_2}\langle\eta_\eps,(\omega-m\sigma_3)\psi_\eps\rangle\,dx\right\vert\leq C \Vert\eta_\eps\Vert_{L^{\frac{4}{3}}(B_2)}\Vert\psi_\eps\Vert_{L^{4}(B_2)}\\
&\leq C\Vert\eta_\eps\Vert_{W^{1,\frac{5}{4}}(B_2)}\Vert\psi_\eps\Vert_{L^{4}(B_2)}\leq C\Vert\psi_\eps\Vert_{L^{4}(B_2)}\Vert\psi_\eps\Vert_{L^{\frac{5}{4}}(B_2)}\,.
\end{split}
\]
Since $\psi_\eps(x):=\theta(x)\Psi(x/\eps)$
\[
\Vert\psi_\eps\Vert_{L^{4}(B_2)}\Vert\psi_\eps\Vert_{L^{\frac{5}{4}}(B_2)}\leq\eps^{\frac{21}{10}}\Vert\Psi\Vert_{L^4}\Vert\Psi\Vert_{L^\frac{5}{4}}\,,
\]
so that 
\[
\left\vert\int_{\R^2}\langle\eta_\eps,(\omega-m\sigma_3)\psi_\eps\rangle\,dx\right\vert=o(\eps^2)\,.
\]
Moreover, there holds
\[
\begin{split}
\int_{\R^2}\langle\psi_\eps,(\omega-m\sigma_3)\psi_\eps\rangle\rangle\,dx&=\eps^2\int_{\R^2}\langle\Psi,(\omega-m\sigma_3)\Psi\rangle\rangle\,dx-\mathcal{O}\left(\eps^2\int_{\R^2\setminus B_{\frac{1}{\eps}}}\vert\Psi\vert^2\,dx\right)\\
&=\eps^2\int_{\R^2}\langle\Psi,(\omega-m\sigma_3)\Psi\rangle\rangle\,dx+o(\eps^2)\,,
\end{split}
\]
as $\Psi\in L^2$ (see \eqref{eq:Sbubbles}). Combining the above observation and \eqref{eq:first-term} we get
\[
\int_{\R^2}\Re\langle\varphi_\eps,A_\omega\varphi_\eps\rangle\,dx\geq4\eps(2S+1)\pi+\eps^2\int_{\R^2}\langle\Psi,(\omega-m\sigma_3)\Psi\rangle\rangle\,dx+o(\eps^3)\,,
\]
and then
\be\label{eq:denom-bound}
\left(\int_{\R^2}\Re\langle\varphi_\eps,A_\omega\varphi_\eps\rangle\,dx\right)^2\geq4^2\eps^2(2S+1)^2\pi^2+\eps^3 2^5(2S+1)^2\pi^2\int_{\R^2}\langle\Psi,(\omega-m\sigma_3)\Psi\rangle\rangle\,dx+o(\eps^3)
\ee
Assume
\be\label{eq:assumption}
M:=\int_{\R^2}\langle\Psi,(\omega-m\sigma_3)\Psi\rangle\rangle\,dx>0\,.
\ee
Then, by \eqref{eq:num-bound} and \eqref{eq:denom-bound} we find
\[
\begin{split}
J(\varphi_\eps) &\leq \frac{1}{4}\frac{4^3(2S+1)^3\pi^3 \eps^{2}+o(\eps^{\frac{10}{3}})}{4^2(2S+1)^2\pi^2 \eps^{2}+2^5(2S+1)^2\pi^2M\eps^3+o(\eps^3)}=(2S+1)\pi\frac{1+o(\eps^{\frac{4}{3}})}{1+2\eps M+o(\eps)}\\
& < (2S+1)\pi\,,
\end{split}
\]
for $\eps>0$ small, thus proving \eqref{eq:inequality}.
\medskip

Suppose now $M<0$ (see \eqref{eq:assumption}). In this case we modify the test spinor \eqref{eq:testspinor} and set
\[
\varphi_\eps:=\theta(x)\sigma_3\Psi(x/\eps)\,,
\]
where $\sigma_3$ is the third Pauli matrix, see \eqref{eq:pauli}. Observe that $\sigma_3$ is hermitian, unitary and anti-commutes with $\cD$, that is
\[
\cD\sigma_3=-\sigma_3\cD\,,
\]
so that $\cD(\sigma_3\Psi)=-\vert\Psi\vert^2\sigma_3\Psi$, where $\Psi$ is one of the bubbles in \eqref{eq:Sbubbles}. It is not hard to see that \eqref{eq:num-bound} still holds. Concerning the denominator in \eqref{eq:quotient}, observe that
\[
\begin{split}
\int_{\R^2}\Re\langle\varphi_\eps,\psi_\eps\rangle\,dx= -\int_{\R^2}\eps^{-1}\theta^2\vert\Psi(x/\eps)\vert^4\,dx&=-\eps^2\int_{\R^2}\vert\Psi\vert^4\,dx+\mathcal{O}\left(\eps^2\int_{\R^2\setminus B_{\frac{1}{\eps}}} \vert\Psi\vert^4\,dx\right) \\
&=-\eps^2\int_{\R^2}\vert\Psi\vert^4\,dx+o(\eps^2)\,,
\end{split}
\]
as $\Psi\in L^2$. Moreover, we still have
\[
\int_{\R^2}\langle\psi_\eps,(\omega-m\sigma_3)\psi_\eps\rangle\rangle\,dx=\eps^2\underbrace{\int_{\R^2}\langle\Psi,(\omega-m\sigma_3)\Psi\rangle\rangle\,dx}_{=M<0}+o(\eps^2)\,,
\]
and 
\[
\left(\int_{\R^2}\Re\langle\varphi_\eps,A_\omega\varphi_\eps\rangle\,dx\right)^2=4^2\eps^2(2S+1)^2\pi^2-2^5(2S+1)^2\pi^2M\eps^3+o(\eps^3)\,,
\]
so that, similarly to the previous case, we get
\[
J(\varphi_\eps) \leq (2S+1)\pi\frac{1+o(\eps^{\frac{4}{3}})}{1-2\eps M+o(\eps)}< (2S+1)\pi\,,
\]
for $\eps>0$ small, as now $M<0$.
\end{proof}
\begin{proposition}\label{prop:existence}
Let $S\in\Z$, $S\neq0$. Then the functional $\cL_\omega$ has a non-trivial critical point in the subspace $E_S\subseteq H^{1/2}(\R^2,\C^2)$ of spinors of the form \eqref{eq:ansatz}. Such spinor is a weak solution to \eqref{criticaldirac}.
\end{proposition}
\begin{proof}
Let $c$ be the minimax level of the dual functional $\cL^*_\omega$ defined in \eqref{eq:mplevel}. By Lemma \ref{lem:inequality} there holds $c<\beta(S)$, so that Lemma \ref{lem:compactness} and a standard deformation argument (see, for instance, \cite[Theorem 3.4]{struwevariational}) give the existence of a critical point for $\cL^*_\omega$. Moreover, by Lemma \ref{lem:correspondence} this corresponds to a critical point of $\cL_\omega$, which in turn is a weak solution to \eqref{criticaldirac} as the latter is the Euler-Lagrange equation of the functional.
\end{proof}
\subsection{Regularity}
Since the nonlinearity in \eqref{criticaldirac} is Sobolev-critical, regularity does not follow by standard arguments and one needs a more refined bootstrap argument, as in \cite{borrellifrank}. To our knowledge, the basic idea behind that proof can be traced back to \cite{jannellisolimini}.

Observe that
\[
(\cD+m\sigma_3-\omega)(\cD+m\sigma_3+\omega)=\begin{pmatrix} -\Del+m^2-\omega^2 &0 \\ 0 & -\Del+m^2-\omega^2 \end{pmatrix}\,.
\]

\begin{lemma}\label{lem:liouville}
Fix $p\geq1$. Let $\psi\in L^{p}(\R^2,\C^2)$ be a distributional solution to 
\be\label{eq:lineareq}
(\cD+m\sigma_3-\omega)\psi=0\,,\qquad m>0, \omega\in(-m,m)\,.
\ee
Then $\psi\equiv 0$.
\end{lemma}
\begin{proof}
Let $\psi\in L^{p}(\R^{n},\C^N)$ be a distributional solution to \eqref{eq:lineareq}, i.e.,
\be\label{eq:weaklinear}
\int_{\R^2}\langle\psi,(\cD+m\sigma_3-\omega)\chi\rangle \,dx=0\,,\qquad\forall\chi\in C^{\infty}_{c}(\R^2,\C^2)\,. 
\ee
Then $\psi$ is also a distributional solution to $(-\Del+\mu^2)\psi=0$, with $\mu^2=m^2-\omega^2$, as
\[
\int_{\R^2}\langle\psi,(-\Del+\mu^2)\chi\rangle \,dx=\int_{\R^2}\langle\psi,(\cD+m\sigma_3-\omega)\underbrace{[(\cD+m\sigma_3+\omega)\chi]}_{\in C^\infty_c(\R^2,\C^2)}\rangle \,dx=0\,,\qquad\forall\chi\in C^{\infty}_{c}(\R^2,\C^2)\,.
\]
Then the claim follows by \cite[Lemma 9.11]{liebloss}.
\end{proof}

We use the above lemma to rewrite \eqref{criticaldirac} as an integral equation. The \emph{Green's function} $\Gamma$ of the Dirac operator $\cD$ is given by
\be\label{eq:greendirac}
\Gamma(x-y)=(\cD_x+m\sigma_3+\omega)G(x-y) \,,\qquad x,y\in\R^2, x\neq y\,,
\ee
where $G(x-y)$ is the Green's function of the operator $( -\Del+\mu^2 )$, with $\mu^2=m^2-\omega^2$. One easily checks that this function satisfies for each fixed $y\in\R^2$ the equation
\be\label{eq:greenequation}
(\cD_x+m\sigma_3+\omega)\Gamma(x-y)=\delta(x-y)I_{2}
\qquad\text{in}\ \R_x^2
\ee
in the sense of distributions. The function $G(x-y)$ is given by
\[
G(x-y)=\frac{1}{2\pi}K_0(\mu\vert x-y\vert)\,,\qquad x,y\in\R^2, x\neq y\,,
\] 
with $K_0$ denoting the inverse Fourier transform of $(\vert \xi\vert^2+\mu^2)^{-1}$ , i.e., the modified Bessel function of second kind
of order 0 \cite[Section 9.6]{abramowitzstegun}. Then one sees that 
\be\label{eq:smallargument}
\Gamma(x)\sim\vert x\vert^{-1}\,,\qquad \mbox{as $x\to0$}
\ee 
and 
\be\label{eq:largeargument}
\vert\Gamma(x)\vert\sim \vert K'_0(x)\vert\sim\sqrt{\frac{\pi}{2x}}e^{-x}\,,\qquad \mbox{as $\vert x\vert\to\infty$}\,,
\ee
 so that it belongs to the \emph{weak}-$L^2$ space, 
\be\label{eq:weakL^2}
\Gamma\in L^{2,\infty}(\R^2,\C^2)\,.
\ee

\begin{lemma}\label{eq:integraleq}
If $\psi\in L^{4}(\R^2,\C^2)$ solves \eqref{criticaldirac} in the sense of distributions, then
\be\label{eq:integralform}
\psi=\Gamma\ast(\vert\psi\vert^2\psi) \,.
\ee
\end{lemma}

\begin{proof}
Since $\psi\in L^4$, $\vert\psi\vert^2\psi\in L^{\frac{4}{3}}$ and therefore, by the weak Young inequality,
$$
\tilde\psi:=\Gamma\ast(\vert\psi\vert^2\psi)
$$
satisfies
$$
\tilde\psi\in L^{4}(\R^2,\C^2) \,,
$$ 
as  $\Gamma\in L^{2,\infty}$. Moreover, it is easy to see that
$$
(\cD+m\sigma_3-\omega)\tilde\psi = \vert\psi\vert^2\psi
\qquad\text{in}\ \R^2
$$
in the sense of distributions. This implies that
$$
(\cD+m\sigma_3-\omega)(\psi-\tilde\psi)=0
\qquad\text{in}\ \R^2
$$
in the sense of distributions and therefore, by Lemma \ref{lem:liouville}, $\psi-\tilde\psi\equiv0$, as claimed.
\end{proof}

\begin{proposition}\label{prop:regularity}
Any distributional solution $\psi\in  L^4(\R^2,\C^2)$ to \eqref{criticaldirac} is smooth.
\end{proposition}
\begin{proof}
Notice that the nonlinearity in \eqref{criticaldirac} is smooth, so that we only need to show that $\psi\in L^\infty(\R^2,\C^2)$. Then smoothness follows by standard elliptic regularity theory.
\smallskip

We first prove that 
\be\label{eq:claim}
\psi\in L^r(\R^2,\C^2)\,, \qquad \mbox{for all $4\leq r<\infty$.} 
\ee
We claim that there exists $C>0$ such that for all $M>0$ there holds
\be\label{eq:sup}
S_M:=\sup\left\{\left\vert\int_{\R^2}\langle\psi,\varphi\rangle\, dx\right\vert\,:\, \Vert \varphi\Vert _{r'}\leq1\,,\,\Vert \varphi\Vert_{4/3}\leq M \right\}\leq C\,,
\ee
so that 
\[
\sup\left\{\left\vert\int_{\R^2}\langle\psi,\varphi\rangle\,d x\right\vert\,:\, \Vert \varphi\Vert _{r'}\leq1\,,\,\varphi\in L^{4/3} \right\}\leq C\,,
\]
and by density and duality, $u\in L^r$. 

Fix $M>0$ and let $\eps>0$ to be determined later. Notice that for any $0< \delta\leq\mu$
\[
f_\eps:=\vert\psi\vert^2\mathbbm{1}_{\{\delta\leq\vert\psi\vert\leq\mu\}}
\]
is bounded and supported on a set of finite measure. We have
\[
\Vert\vert\psi\vert^2-f_\eps \Vert^2_{2}=\int_{\{\vert\psi\vert<\delta\}\cup\{\vert\psi\vert>\mu\}}\vert\psi\vert^4\dd x<\eps
\]
for suitable $\delta,\mu>0$, since $\psi\in L^{4}$. Set $g_\eps:=\vert\psi\vert^2-f_\eps$ and consider $\varphi\in L^{r'}\cap L^{4/3}$ such that $\Vert v\Vert_{r'}\leq 1$ and $\Vert v\Vert_{4/3}\leq M$.

Then \eqref{eq:integralform} gives
\[
\int_{\R^2}\langle\psi,\varphi\rangle\,dx=\int_{\R^2}\langle(\Gamma\ast(f_\eps \psi)),\varphi\rangle\,dx +\int_{\R^2}\langle(\Gamma\ast(g_\eps \psi)),\varphi\rangle\,dx\,.
\]
Fubini's theorem allows to rewrite the second integral on the right-hand side :
\be
\begin{split}
\int_{\R^2}\langle\Gamma\ast(g_\eps \psi),\varphi\rangle\,dx &= \int_{\R^2}\,dx\,\langle\int_{\R^2}\Gamma(x-y)(g_\eps(y) \psi(y))\dd y,\varphi(x)\rangle \\
&=\int_{\R^2}d y\int_{\R^2}\, dx \langle g_\eps(y)\psi(y),\Gamma(x-y)\varphi(x)\rangle \\
&= \int_{\R^2}\langle g_\eps\psi,\Gamma\ast\varphi\rangle\,dy\,.
\end{split}
\ee
Arguing similarly, recalling that $\psi=\Gamma\ast(\vert\psi\vert^2\psi)$, we the last integral can be rewritten so that
\be\label{eq:split}
\int_{\R^2}\langle\psi,\varphi\rangle \,dx=\int_{\R^2}\langle\Gamma\ast(f_\eps \psi),\varphi\rangle\,dx+\int_{\R^2}\langle\psi,\chi_\eps\rangle\,dx\,,
\ee
where
\be\label{eq:h}
\chi_\eps:=\vert\psi\vert^2 \Gamma\ast(g_\eps(\Gamma\ast \varphi))\,.
\ee
Define $s:=\frac{2r}{2+r}$. Then we can now estimate the first integral in \eqref{eq:split} using the H\"older and Young inequalities
\be\label{eq:firstbound}
\begin{split}
\left\vert \int_{\R^2}\langle\Gamma\ast(f_\eps u),\varphi\rangle\,dx  \right\vert &\leq\Vert \Gamma\ast(f_\eps \psi)\Vert_{r}\Vert \varphi\Vert_{r'}\leq\Vert \Gamma\Vert_{2,\infty}\Vert f_\eps \psi\Vert_{s}\Vert \varphi\Vert_{r'}\\
&\leq \Vert G\Vert_{3,\infty}\Vert f_\eps\Vert_{\frac{4s}{4-s}}\Vert u\Vert_6\Vert \varphi\Vert_{r'}\leq C_\eps\,.
\end{split}
\ee
Notice that the constant $C_\eps$ does not depend on M, but only on $\eps,r,\psi$.

The second integral on the right-hand side of \eqref{eq:split} can be bounded as follows. By \eqref{eq:h} and H\"older and Young inequalities, we get
\be\label{eq:first_h_estimate}
\begin{split}
\Vert \chi_\eps\Vert_{r'}&\leq\Vert \vert\psi\vert^2\Vert_{2}\Vert \Gamma\ast(g_\eps(\Gamma\ast \varphi)) \Vert_{s'}\leq \Vert\vert\psi\vert^2\Vert_{2}\Vert \Gamma\Vert_{2,\infty}\Vert g_\eps(\Gamma\ast \varphi)\Vert_{r'} \\
&\leq\Vert\vert\psi\vert^2\Vert_{2}\Vert \Gamma\Vert_{2,\infty}\Vert g_\eps\Vert_{2}\Vert \Gamma\ast \varphi\Vert_{s'}\leq \Vert\vert\psi\vert^2\Vert_{2}\Vert \Gamma\Vert^2_{2,\infty}\Vert g_\eps\Vert_{2}\Vert \varphi\Vert_{r'} \\
&\leq C' \Vert g_\eps\Vert_{2}\Vert \varphi\Vert_{r'}\,,
\end{split}
\ee
 the constant $C'>0$ depending on $\psi$. Similarly, we get
\be\label{eq:second_h_estimate}
\begin{split}
\Vert \chi_\eps\Vert_{4/3}&\leq\Vert\vert\psi\vert^2\Vert_{2}\Vert \Gamma\ast(g_\eps(\Gamma\ast \varphi))\Vert_4\leq\Vert\vert\psi\vert^2\Vert_{2}\Vert \Gamma\Vert_{2,\infty}\Vert g_\eps(\Gamma\ast \varphi)\Vert_{4/3} \\
&\leq\Vert\vert\psi\vert^2\Vert_{2}\Vert \Gamma\Vert_{3,\infty}\Vert g_\eps\Vert_{2}\Vert(G\ast \varphi)\Vert_{4}\leq\Vert\vert\psi\vert^2\Vert_{2}\Vert \Gamma\Vert^2_{2,\infty}\Vert g_\eps\Vert_{2}\Vert \varphi\Vert_{4/3}\\
&\leq C'\Vert g_\eps\Vert_{2}\Vert \varphi\Vert_{4/3}\,.
\end{split}
\ee
Estimates \eqref{eq:first_h_estimate} and \eqref{eq:second_h_estimate} give
\[
\left\vert \int_{\R^2}\langle\psi,\chi_\eps\rangle\,dx\right\vert\leq C'\Vert g_\eps\Vert_{3/2}S_M\leq C'\eps S_M\,,
\]
by \eqref{eq:sup}. Choosing $\eps=(2C')^{-1}$ and taking into account \eqref{eq:firstbound} we obtain
\[
\left\vert\int_{\R^2}\langle\psi,\varphi\rangle\,d x \right\vert\leq C''+\frac{1}{2}S_M\,,
\]
where $C''$ equals the constant $C_\eps$ for $\eps=(2C')^{-1}$. Taking the supremum over all $\varphi$ we have
\[
S_M\leq C''+\frac{1}{2}S_M\implies S_M\leq 2C''\,,
\]
thus proving \eqref{eq:claim}. Recall that by \eqref{eq:integralform}, we have
\[
\psi(x)=\int_{\R^2}\Gamma(x-y)\vert\psi(y)\vert^2\psi(y)\dd y\,,
\]
and since we can write $\Gamma$ as the sum of a function in $L^p$ and one in $L^q$, for some $1<p<3<q$, by the H\"older inequality we get $\psi\in L^\infty$.
\end{proof}

\subsection{Exponential decay of solutions} We conclude the proof of Theorem \eqref{thm:main} showing that the solutions to \eqref{criticaldirac} found in Section \eqref{sec:existence} have exponential decay at infinity. Let $\psi\in L^4(\R^2,\C^2)$ be a distributional solution to \eqref{criticaldirac}, then by Proposition \ref{prop:regularity} such solution is also smooth. Moreover, the proof of such result shows that $\psi\in L^p$ for all $p\in[4,\infty]$ so that, using the properties of the Green function $\Gamma$ in \eqref{eq:greendirac} we can also prove that $\psi$ is H\"older continuous. 
\begin{lemma}\label{lem:holder}
Let $\psi\in L^4(\R^2,\C^2)$ be a distributional solution to \eqref{criticaldirac}. Then $\psi\in C^{0,\alpha}$ for some $\alpha\in(0,1)$
\end{lemma}
\begin{proof}
Let $x,z\in\R^2$ with $x\neq z$. Then by \eqref{eq:integralform} we get
\[
\vert\psi(x)-\psi(z)\vert=\left\vert\int_{\R^2}(\Gamma(x-y)-\Gamma(z-y))\vert\psi(y)\vert^2\psi(y)\,dy \right\vert
\]
Take $r=\vert x-z\vert$, and split the above integral as follows
\be\label{eq:splitintegral}
\begin{split}
&\left\vert\int_{B_{2r}(x)}(\Gamma(x-y)-\Gamma(z-y))\vert\psi(y)\vert^2\psi(y)\,dy \right\vert \\
&+\left\vert\int_{\R^2\setminus B_{2r}(x)}(\Gamma(x-y)-\Gamma(z-y))\vert\psi(y)\vert^2\psi(y)\,dy \right\vert=: I+ II\,.
\end{split}
\ee
By the choice of the radius $r$ we see that $B_{2r}(x)\subseteq B_{3r}(z)$ so that the first term can be estimated as
\be\label{eq:first}
I\leq \int_{B_{2r}(x)}\vert\Gamma(x-y)\vert\vert\psi(y)\vert^2\psi(y)\,dy+\int_{B_{3r}(z)}\vert\Gamma(z-y)\vert\vert\psi(y)\vert^2\psi(y)\,dy\,.
\ee
Then, since $\psi\in L^s$ for all $s\in[4,\infty]$ and using \eqref{eq:smallargument} we get
\[
\begin{split}
 \int_{B_{2r}(x)}\vert\Gamma(x-y)\vert\vert\psi(y)\vert^2\psi(y)\,dy&\leq C\int_{B_{2r}(x)}\vert x-y\vert^{-1}\vert\psi(y)\vert\,dy \\
 &\leq C\Vert\psi\Vert_{L^p}\left( \int_{B_{2r}(x)}\vert x-y\vert^{-\frac{p}{p-1}}\,dy\right)^{\frac{p-1}{p}}\\
 & \leq C\Vert\psi\Vert_{L^p}\left( \int_{B_{2r}(0)}\vert w\vert^{-\frac{p}{p-1}}\,dw\right)^{\frac{p-1}{p}}\\
 &\leq C \Vert\psi\Vert_{L^p} r^{\frac{p-2}{p-1}}\,,
\end{split}
\]
where we have used the H\"older inequality with $p\geq 4$ . The other integral in \eqref{eq:first} can be estimated similarly, and thus
\be\label{eq:firstestimate}
I\leq C\Vert\psi\Vert_{L^p}r^{\frac{p-2}{p-1}}\,, \qquad p\geq4\,.
\ee
We now turn to the second integral in \eqref{eq:splitintegral}. Observe that if $y\in\R^2\setminus B_{2r}(x)$, then $y\neq x$ and $y\neq z$ so that we can apply the mean value theorem and get the existence of a point $w_y$ on the segment between $y-x$ and $z-y$ so that 
\[
\vert \Gamma(x-y)-\Gamma(z-y)\vert\leq \vert\nabla\Gamma(w_y)\vert\vert x-z\vert\,.
\]
By \eqref{eq:largeargument} we see that 
\[
\vert\nabla\Gamma(w_y)\vert\leq C\vert y-x\vert^{-2}\,,\qquad y\in\R^2\setminus B_{2r}(x)\,,
\]
 and then, arguing as for \eqref{eq:firstestimate}, we can estimate
 \be\label{eq:secondestimate}
 \begin{split}
 II&\leq Cr\int_{\R^2\setminus B_{2r}(x)}\vert x-y\vert^{-2}\vert\psi(y)\vert\,dy\\
 &\leq Cr\Vert\psi\Vert_{L^q}\left(\int_{\R^2\setminus B_{2r}(0)}\vert w\vert^{-2}\,dw \right)^{\frac{q-1}{q}}\\
 &\leq C\Vert\psi\Vert_{L^q} r^{3-2\frac{q}{q-1}}\,,
 \end{split}
 \ee
 where $q\geq 4$. Then the claim follows combining \eqref{eq:firstestimate} and \eqref{eq:secondestimate}, by the arbitrariness of $p,q\geq4$.
\end{proof}

The above result implies, in particular, that $\psi$ is uniformly continuous and thus tends to zero at infinity.
\begin{lemma}
Let $\psi\in L^p(\R^2,\C^2)\cap C^{0,\alpha}(\R^2,\C^2)$ for some $p\geq 1$, $0<\alpha<1$. Then
\be\label{eq:tendstozero}
\lim_{\vert x\vert\to\infty}\psi(x)=0\,.
\ee
\end{lemma}
\begin{proof}
Suppose \eqref{eq:tendstozero} does not hold. Thus there exist $\eps>0$ and a sequence of points $(x_n)_n\subseteq\R^2$, with $\lim_{n\to\infty}\vert x_n\vert=\infty$, such that
\[
\vert\psi(x_n)\vert\geq\eps\,,\qquad \forall n\in\N\,.
\]
By uniform continuity there exists $\delta>0$ such that 
\[
\vert\psi(x)\vert\geq\frac{\eps}{2}\,,\qquad\mbox{if $\vert x-x_n\vert<\delta$,}
\]
so that
\[
\int_{\vert x-x_n\vert<\delta}\vert\psi\vert^p\, dx \geq\frac{\eps^p\delta}{2^p}\,,\qquad\forall n\in\N\,,
\]
contradicting the fact that $\psi\in L^p$.
\end{proof}
Assume that $\psi$ is a smooth solution to \eqref{criticaldirac} of the form \eqref{eq:ansatz}. Plugging such ansatz into \eqref{criticaldirac} we get the following system for $(u,v)$:
\begin{equation}\label{eq:radial}
\left\{\begin{aligned}
    u'+\frac{S+1}{r}u &=(u^{2}+v^{2})v-(m-\omega)v \\ 
   v'-\frac{S}{r}v&=-(u^{2}+v^{2})u-(m+\omega)u
\end{aligned}\right.
\end{equation}
where the $u':=\frac{du}{dr}$, $v':=\frac{dv}{dr}$. 
\begin{proposition}\label{prop:expdecay}
Let $\psi\in L^4(\R^2,\C^2)\cap C^\infty(\R^2,\C^2)$ be a solution to \eqref{criticaldirac} of the form \eqref{eq:ansatz}. Then there holds $\psi(0)=0$ and 
\be\label{eq:expdecay}
\vert\psi(r,\theta)\vert^2=u^2(r)+v^2(r)\leq C e^{-\sqrt{m-\omega}r}\,,\qquad r>0,\theta\in\mathbb{S}^1\,,
\ee
for some constant $C>0$.
\end{proposition}
\begin{proof}
Define $f:=u^2+v^2$. Observe that the singular terms in \eqref{eq:radial} and the smoothness of $\psi$ imply that $u(0)=v(0)=0$, i.e, $\psi(0)=0$. Moreover, observe that by Lemma \ref{lem:holder} and \eqref{eq:tendstozero} we know that 
\be\label{eq:tozero}
\lim_{r\to\infty}u^2(r)+v^2(r)=0\,.
\ee

By a direct calculation, we get
\[
f'=-\frac{S+1}{r}u^2+\frac{S}{r}v^2-2muv\,,
\]
and
\[
(uv)'=-\frac{uv}{r}+(v^2-u^2)(u^2+v^2)-(m-\omega)v^2-(m+\omega)u^2\,.
\]
Then a straightforward computation gives
\be\label{eq:f''}
f''(r)=2m(m-\omega)v^2(r)+2m(m+\omega)u^2(r)+V(r)(u^2(r)+v^2(r)) \\
\,,
\ee
where $\lim_{r\to\infty}V(r)=0$ by \eqref{eq:tozero}. Given $0<\eps<2m(m-\omega)$, take $R_\eps>0$ such that $V(r)\geq -\eps$ for all $r\geq R_\eps$. Then the function
\[
g(r):=f(r)-f(R_\eps)e^{-\sqrt{2m(m-\omega)-\eps}(r-R_\eps)}\,,\qquad r\geq R_\eps\,,
\]
verifies $g(R_\eps)=0$, $\lim_{r\to\infty}g(r)=0$ and
\[
g''(r)\geq(2m(m-\omega)-\eps)g(r)\,,\qquad r\geq R_\eps\,.
\]
Then the maximum principle gives $g(r)\leq 0$, for all $r\geq R_\eps$, so that we find
\[
f(r)\leq f(R_\eps)e^{-\sqrt{2m(m-\omega)-\eps}(r-R_\eps)}\,,\qquad \mbox{for $r\geq R_\eps$\,.}
\]
By continuity, we conclude that there exists a constant $C_\eps>0$ such that 
\be\label{eq:almost}
f(r)\leq C_\eps e^{-(\sqrt{2m(m-\omega)-\eps})r}\,,\qquad\mbox{for $r>0$.}
\ee
 Equation \eqref{eq:f''} can be rewritten as
\[
-f''(r)+k^2f(r)=G(r)\,, r\geq0\,,
\]
where $k^2=2m(m-\omega)$ and $G(r)=-V(r)[(1+4m\omega)u^2(r)+v^2(r)]$ for which a decay estimate analogous to \eqref{eq:almost} holds. Recall that $f(0)=0$, that is, $f$ verifies Dirichlet boundary conditions, so that applying the Green's function one finds
\[
f(r)=-\frac{1}{2k}\int^\infty_0 G(\rho)\left(e^{-k\vert r-\rho\vert}+e^{-k\vert r+\rho\vert}\right)\,d\rho\,.
\]
Then \eqref{eq:expdecay} easily follows.
\end{proof}


\begin{thebibliography}{10}

\bibitem{abramowitzstegun}
{\sc M.~Abramowitz and I.~A. Stegun}, {\em Handbook of mathematical functions
  with formulas, graphs, and mathematical tables}, vol.~55 of National Bureau
  of Standards Applied Mathematics Series, For sale by the Superintendent of
  Documents, U.S. Government Printing Office, Washington, D.C., 1964.

\bibitem{ammannsmallest}
{\sc B.~Ammann}, {\em The smallest {D}irac eigenvalue in a spin-conformal class
  and cmc immersions}, Comm. Anal. Geom., 17 (2009), pp.~429--479.

\bibitem{spinorialanalogue}
{\sc B.~Ammann, J.-F. Grosjean, E.~Humbert, and B.~Morel}, {\em A spinorial
  analogue of {A}ubin's inequality}, Math. Z., 260 (2008), pp.~127--151.

\bibitem{ammannmass}
{\sc B.~Ammann, E.~Humbert, and B.~Morel}, {\em Mass endomorphism and spinorial
  {Y}amabe type problems on conformally flat manifolds}, Comm. Anal. Geom., 14
  (2006), pp.~163--182.

\bibitem{arbunichsparber}
{\sc J.~Arbunich and C.~Sparber}, {\em Rigorous derivation of nonlinear {D}irac
  equations for wave propagation in honeycomb structures}, J. Math. Phys., 59
  (2018), pp.~011509, 18.

\bibitem{bartschspinorial}
{\sc T.~{Bartsch} and T.~{Xu}}, {\em {A spinorial analogue of the
  Brezis-Nirenberg theorem involving the critical Sobolev exponent}}, ArXiv
  e-prints,  (2018).

\bibitem{shooting}
{\sc W.~Borrelli}, {\em Stationary solutions for the 2{D} critical {D}irac
  equation with {K}err nonlinearity}, J. Differential Equations, 263 (2017),
  pp.~7941--7964.

\bibitem{massless}
\leavevmode\vrule height 2pt depth -1.6pt width 23pt, {\em Weakly localized
  states for nonlinear {D}irac equations}, Calc. Var. Partial Differential
  Equations, 57 (2018), p.~57:155.

\bibitem{borrellifrank}
{\sc W.~Borrelli and R.~L. Frank}, {\em Sharp decay estimates for critical
  {D}irac equations}, Trans. Amer. Math. Soc., 373 (2020), pp.~2045--2070.

\bibitem{Borrelli-Maalaoui-JGA2020}
{\sc W.~Borrelli and A.~Maalaoui}, {\em Some properties of {D}irac-{E}instein
  bubbles}, J. Geometric Analysis,  (2020).

\bibitem{borrellimalchiodiwu}
{\sc W.~{Borrelli}, A.~{Malchiodi}, and R.~{Wu}}, {\em {Ground state Dirac
  bubbles and Killing spinors}}, to appear on Comm. Math. Phys.,  (2020), p.~arXiv:2003.03949.

\bibitem{Brezis-Lieb}
{\sc H.~Br\'{e}zis and E.~Lieb}, {\em A relation between pointwise convergence
  of functions and convergence of functionals}, Proc. Amer. Math. Soc., 88
  (1983), pp.~486--490.

\bibitem{CassanoDecay}
{\sc B.~Cassano}, {\em Sharp exponential decay for solutions of the stationary
  perturbed dirac equation}, Communications in Contemporary Mathematics, DOI:10.1142/S0219199720500674.

\bibitem{dingruf}
{\sc Y.~Ding and B.~Ruf}, {\em Solutions of a nonlinear {D}irac equation with
  external fields}, Arch. Ration. Mech. Anal., 190 (2008), pp.~57--82.

\bibitem{dingwei}
{\sc Y.~Ding and J.~Wei}, {\em Stationary states of nonlinear {D}irac equations
  with general potentials}, Rev. Math. Phys., 20 (2008), pp.~1007--1032.

\bibitem{els}
{\sc M.~J. Esteban, M.~Lewin, and E.~S{\'e}r{\'e}}, {\em Variational methods in
  relativistic quantum mechanics}, Bull. Amer. Math. Soc. (N.S.), 45 (2008),
  pp.~535--593.

\bibitem{es}
{\sc M.~J. Esteban and E.~S\'er\'e}, {\em Stationary states of the nonlinear
  {D}irac equation: a variational approach}, Comm. Math. Phys., 171 (1995),
  pp.~323--350.

\bibitem{FWhoneycomb}
{\sc C.~L. Fefferman and M.~I. Weinstein}, {\em Honeycomb lattice potentials
  and dirac points}, J. Amer. Math. Soc., 25 (2012), pp.~1169--1220.

\bibitem{wavedirac}
\leavevmode\vrule height 2pt depth -1.6pt width 23pt, {\em Waves in honeycomb
  structures}, Journ\'ees \'equations aux d\'eriv\'ees partielles,  (2012).

\bibitem{FWwaves}
\leavevmode\vrule height 2pt depth -1.6pt width 23pt, {\em Wave packets in
  honeycomb structures and two-dimensional {D}irac equations}, Comm. Math.
  Phys., 326 (2014), pp.~251--286.

\bibitem{nadineconformalinvariant}
{\sc N.~Grosse}, {\em On a conformal invariant of the {D}irac operator on
  noncompact manifolds}, Ann. Global Anal. Geom., 30 (2006), pp.~407--416.

\bibitem{nadineboundedgeometry}
\leavevmode\vrule height 2pt depth -1.6pt width 23pt, {\em Solutions of the
  equation of a spinorial {Y}amabe-type problem on manifolds of bounded
  geometry}, Comm. Partial Differential Equations, 37 (2012), pp.~58--76.

\bibitem{isobecritical}
{\sc T.~Isobe}, {\em Nonlinear {D}irac equations with critical nonlinearities
  on compact {S}pin manifolds}, J. Funct. Anal., 260 (2011), pp.~253--307.

\bibitem{jannellisolimini}
{\sc E.~Jannelli and S.~Solimini}, {\em Concentration estimates for critical
  problems}, Ricerche Mat., 48 (1999), pp.~233--257.
\newblock Papers in memory of Ennio De Giorgi (Italian).

\bibitem{liebloss}
{\sc E.~H. Lieb and M.~Loss}, {\em Analysis}, vol.~14 of Graduate Studies in
  Mathematics, American Mathematical Society, Providence, RI, second~ed., 2001.

\bibitem{CCcompactI}
{\sc P.-L. Lions}, {\em The concentration-compactness principle in the calculus
  of variations. {T}he locally compact case. {I}}, Ann. Inst. H. Poincar\'e
  Anal. Non Lin\'eaire, 1 (1984), pp.~109--145.

\bibitem{CCcompactII}
\leavevmode\vrule height 2pt depth -1.6pt width 23pt, {\em The
  concentration-compactness principle in the calculus of variations. {T}he
  locally compact case. {II}}, Ann. Inst. H. Poincar\'e Anal. Non Lin\'eaire, 1
  (1984), pp.~223--283.

\bibitem{CClimitI}
\leavevmode\vrule height 2pt depth -1.6pt width 23pt, {\em The
  concentration-compactness principle in the calculus of variations. {T}he
  limit case. {I}}, Rev. Mat. Iberoamericana, 1 (1985), pp.~145--201.

\bibitem{maalaoui}
{\sc A.~Maalaoui}, {\em Infinitely many solutions for the spinorial {Y}amabe
  problem on the round sphere}, NoDEA Nonlinear Differential Equations Appl.,
  23 (2016), pp.~Art. 25, 14.

\bibitem{Maalaoui-Martino-JDE19}
{\sc A.~Maalaoui and V.~Martino}, {\em Characterization of the {P}alais-{S}male
  sequences for the conformal {D}irac-{E}instein problem and applications}, J.
  Differential Equations, 266 (2019), pp.~2493--2541.

\bibitem{criticalhamiltonian}
{\sc J.~Mawhin and M.~Willem}, {\em Critical point theory and {H}amiltonian
  systems}, vol.~74 of Applied Mathematical Sciences, Springer-Verlag, New
  York, 1989.

\bibitem{Palatucci-Pisante}
{\sc G.~Palatucci and A.~Pisante}, {\em Improved {S}obolev embeddings, profile
  decomposition, and concentration-compactness for fractional {S}obolev
  spaces}, Calc. Var. Partial Differential Equations, 50 (2014), pp.~799--829.

\bibitem{raulot}
{\sc S.~Raulot}, {\em A {S}obolev-like inequality for the {D}irac operator}, J.
  Funct. Anal., 256 (2009), pp.~1588--1617.

\bibitem{struwevariational}
{\sc M.~Struwe}, {\em Variational methods}, vol.~34 of Ergebnisse der
  Mathematik und ihrer Grenzgebiete. 3. Folge. A Series of Modern Surveys in
  Mathematics [Results in Mathematics and Related Areas. 3rd Series. A Series
  of Modern Surveys in Mathematics], Springer-Verlag, Berlin, fourth~ed., 2008.
\newblock Applications to nonlinear partial differential equations and
  Hamiltonian systems.

\bibitem{diracthaller}
{\sc B.~Thaller}, {\em The {D}irac equation}, Texts and Monographs in Physics,
  Springer-Verlag, Berlin, 1992.

\end{thebibliography}
\end{document}